\documentclass[12pt]{article}
\usepackage{amsmath,amsthm,amssymb,amsfonts,latexsym,mathrsfs,bm,tikz-cd,ytableau}
\usepackage{hyperref}
\usepackage{url}
\usepackage{setspace}
\usepackage{ytableau}
\usepackage{tikz}

\newtheorem{theorem}{Theorem}[section]
\newtheorem{lemma}{Lemma}[section]

\newtheorem{definition}{Definition}[section]
\newtheorem{corollary}{Corollary}[section]
\newtheorem{proposition}{Proposition}[section]
\newtheorem{remark}{Remark}[section]
\newtheorem*{theoremA}{Theorem A}
\usepackage[a4paper]{geometry}                           
\title{The cohomology of certain intermediate strata of Kottwitz varieties}
\author{Yachen Liu}

\begin{document}	
\maketitle

\begin{abstract}
We derive explicit formulas for the Frobenius-Hecke traces of the etale cohomology of certain strata of Kottwitz varieties (which are certain compact unitary type Shimura varieties considered by Kottwitz), in terms of automorphic representations and certain explicit polynomials. We obtain our results using the trace formula, representations of general linear groups over p-adic fields, and a truncation of the formula of Kottwitz for the number of points on Shimura varieties over finite fields.

\end{abstract}
\newpage	
\tableofcontents
	
\section{Introduction}
Let $(G,X)$ be a PEL-type Shimura datum, as defined in $\S4$ of \cite{kot92p}, and let $E$ denote its reflex field (for definitions of the Shimura datum and its reflex field, see \cite{Del79}). Let $p$ be a prime number such that there exists a flat model $\mathcal G$ of $G_{\mathbb Q_p}$ defined over $\mathbb Z_p$, define $K_p$ as the group $\mathcal G(\mathbb Z_p)$ and let $K^p \subset G(\mathbb A_f^p)$ be a neat compact open subgroup. In $\S 5$ of \cite{kot92p}, Kottwitz constructs a PEL-type moduli problem of abelian varieties over $\mathcal O_E \otimes_{\mathbb Z} \mathbb Z_{(p)}$ using the PEL-type Shimura datum along with some additional data (which we will not discuss, as they will not be used in this paper), we will denote the PEL-type moduli problem by $\mathcal M$, the moduli problem $\mathcal M$ is represented by a smooth quasi-projective scheme over $\mathcal O_E \otimes_{\mathbb Z} \mathbb Z_{(p)}$ (loc. cit.). Let $S$ denote the moduli space associated with $\mathcal M$, and let $M_{K_pK^p}$ denote the canonical model of $Sh_{K_pK^p}(G,X)$ over $E$, as defined in $\S$2.2 of \cite{Del79}. Then the scheme $S_E$ is isomorphic to $|\mathrm {ker}^1(\mathbb Q, G)|$ copies of $M_{K_pK^p}$ (see $\S$8 of \cite{kot92p}). 
 
 Let $v$ be a place of $E$ above $p$ and let $\pi_{E_v}$ be a uniformizer in $\mathcal O_{E_v}$. We write $q = |\mathcal O_{E_v} / (\pi_{E_v})|$ and let $\mathbb F_q = \mathcal O_{E_v} / (\pi_{E_v})$. We extend $S$ over $\mathcal O_{E_v}$ via the integral canonical model (as defined in \cite{Kis10}), and we use $\bar S$ to denote $S_{\mathbb F_q}$. Let $B(G_{\mathbb Q_p})$ denote the set of isocrystals with additional $G$-structure (as defined in \cite{Kot85}), which is the set of the $\sigma$-conjugation class in $G(L)$, where $L$ is the completion of the maximal unramified extension of $\mathbb Q_p$, and $\sigma: G(L) \rightarrow G(L)$ is the Frobenius automorphism. 
 For each geometric point $x \in \bar S(\overline {\mathbb F}_q)$, let $b(\mathcal A_x)$ denote the isocrystal with additional $G$-structure associated to the abelian variety $\mathcal A_x$ corresponding to $x$ (see pp.170-171 of \cite{RaRi96}). The set of isocrystals with additional $G$-structure $\{b(\mathcal A_x): x \in \bar S(\overline {\mathbb F}_q)\}$ 
 is contained in a certain finite set $B(G_{\mathbb Q_p} , \mu) \subset B(G_{\mathbb Q_p})$ (see $\S 6$ of \cite{Kot97}). For $b \in B(G_{\mathbb Q_p} )$, we define $S_b$ to be $\{x \in \bar S(\overline {\mathbb F}_q): b(\mathcal A_x) = b\}$. Then $S_b$ is a locally closed subset of $\bar S_{\overline {\mathbb F}_q}$ (see $\S 1$ of \cite{Ra02}), and we equip $S_b$ with the structure of an algebraic variety induced from $S_b$. The collection $\{S_b\}_{b\in B(G_{\mathbb Q_p},\mu)}$ of subvarieties of $\bar S$ is called the Newton stratification of $\bar S$ (for properties of the Newton stratification $\{S_b\}_{b\in B(G_{\mathbb Q_p},\mu)}$, see $\S 7$ of \cite{Ra02}). 

When $b$ is the unique basic element of $B(G_{\mathbb Q_p},\mu)$ (see $\S 6.4$ of \cite{Kot97}), we call $S_b$ the basic stratum (note that when $b $ is basic, the $S_b$ is closed in $\bar S_{\overline{\mathbb F}_q}$, see $\S 1$ of \cite{Ra02}). The basic stratum has been studied by many people using different methods. For geometric methods, see \cite{Kob}, \cite{KO85}, \cite{LO98} for Siegel modular varieties, and see \cite{Vol10}, \cite{Vol11}, \cite{Ho14} and \cite{Fox21} for quasi-split unitary Shimura varieties. For the study of the basic stratum of Kottwitz varieties (introduced in $\S 1$ of \cite{Kot92}) using automorphic forms, see \cite{Kret11}, \cite{Kret15}, \cite{Kret17} and \cite{Yachen24}. When $b$ is the unique $\mu$-ordinary element of $B(G_{\mathbb Q_p},\mu)$ (see $\S 6$ of \cite{Kot97}), then $S_b$ is dense in $\overline S_{\overline{\mathbb F}_q}$ (see \cite{Wed99}).

 Let $D$ be a division algebra over $\mathbb Q$, with an anti-involution $*$ such that $*$ induces the complex conjugation on $F$ (where $F$ is the center of $D$). In this article, we consider certain compact PEL-type unitary Shimura varieties of signature $(s,n-s)$ arising from $D$, introduced by Kottwitz in $\S 1$ of \cite{Kot92}, which we refer to as Kottwitz varieties (with corresponding Shimura datum $(G,X)$). We assume that $p$ is a split prime number in the extension $F/\mathbb Q$, in this case, we have $G_{\mathbb Q_p} \cong \mathbb G_{m,\mathbb Q_p} \times GL_{n,\mathbb Q_p}$.
Now, let $b \in B(G_{\mathbb Q_p},\mu)$ such that the corresponding Newton vector $\nu(b)$ (as defined in $\S5$ of \cite{Kot85}, see also \cite{Ra02} for the case when $G_{\mathbb Q_p}$ is a general linear group) equals $(s, \lambda_{1}^{p_1},\lambda_2^{p_2}) \in \mathbb R^{n+1}$, where \begin{itemize}
\item $p_1,p_2 \in \mathbb Z_{>0}$ and $p_1 +p_2 = n$,
\item $\lambda_1> \lambda_2>0$ and $(\lambda_1p_1, p_1) = (\lambda_2p_2,p_2) = 1$.
\end{itemize}
Let $S_b$ be the corresponding stratum, we use $\iota:B \hookrightarrow \bar S$ to denote the closed subscheme of $\bar S$ (defined over $\mathbb F_q$), such that $B_{\overline{\mathbb F}_q}=S_b$ (see Theorem 3.6 of \cite{RaRi96}).
Following \cite{Kret11} and \cite{Kret12b}, we use the Langlands-Kottwitz method (which we will introduce in the next paragraph) to study the stratum $S_b$.

 In the case of good reduction, the trace of the Frobenius-Hecke operator of $G(\mathbb A^p)$ on the $\ell$-adic cohomology of the reduction mod $p$ of Shimura varieties can be described in terms of orbital integrals and then in terms of elliptic
parts of the stable Arthur–Selberg trace formula for the endoscopic groups (see \cite{kot90}, \cite{kot92p} for PEL-type Shimura varieties and see \cite{SKZ21} for abelian type Shimura varieties). This approach to studying the Frobenius-Hecke trace is called the Langlands-Kottwitz method. 
In \cite{Kret11} and \cite{Kret12b}, the Langlands-Kottwitz method is adapted to study the Frobenius-Hecke trace of the $\ell$-adic cohomology of Newton strata, and we can express the Frobenius-Hecke trace of Newton strata in terms of orbital integrals. 
In this article, we study the Frobenius-Hecke action on the $\ell$-adic cohomology of $S_b$, and we derive a formula for the Frobenius-Hecke traces in terms of automorphic representations.

We now give a precise statement of the main results:

\begin{theoremA} (See Theorem 6.1) Let $\alpha \in \mathbb Z_{>0}$, then we have: 
\begin{enumerate}
    
\item If $n$ is odd, or $n$ is even and $(p_1, p_2) \not = (\frac n2,\frac n2), (\frac n2 \pm 1, \frac n2 \mp 1)$, then $$\mathrm{Tr}(\mathrm{Frob}^{\alpha}_q\times f^{p \infty},\sum\limits_{i=0}^{\infty}H^i_{et,c}(S_{b},\iota^*\mathcal L)) = |\mathrm{ker}^1|(\mathbb Q, G) \sum\limits_{\substack{\pi \subset \mathcal A(G)\\ \rho_{\pi,p} \text{ is of }\text{ type I}} } c_{\pi}\mathrm{Tr}( C_{\lambda}f_{n\alpha s}, \rho_{\pi,p}) $$
\item If $n$ is even and $(p_1, p_2) = (\frac n2, \frac n2)$ or $(p_1,p_2) = (\frac n2 \pm 1, \frac n2 \mp 1)$, then $$\mathrm{Tr}(\mathrm{Frob}^{\alpha}_q\times f^{p \infty},\sum\limits_{i=0}^{\infty}H^i_{et,c}(S_{b},\iota^*\mathcal L)) = |\mathrm{ker}^1|(\mathbb Q, G) \sum\limits_{\substack{\pi \subset \mathcal A(G)\\ \rho_{\pi,p} \text{ is of }\text{ type II}} } c_{\pi}\mathrm{Tr}( C_{\lambda}f_{n\alpha s}, \rho_{\pi,p}) $$

\end{enumerate}
where
\begin{itemize}
\item the constant $c_\pi$ equals $\zeta_{\pi}^{\alpha} \mathrm{Tr}(f^{p\infty},\pi^{p \infty})N_{\infty}\sum_{i=1}^{\infty}(-1)^idim(H^i(\mathfrak g, K_{\infty};\pi_{\infty}\otimes \zeta))$ such that $\zeta_{\pi}$ is a root of unity. 
\item by assumption, we have $G(\mathbb Q_p) \cong \mathbb Q_p^{\times} \times GL_n(\mathbb Q_p)$. We write $\pi_p \cong \sigma_{\pi,p} \otimes \rho_{\pi,p} $, so that $\sigma_{\pi,p}$ (resp. $\rho_{\pi,p}$) is a representation of $\mathbb Q_p^{\times}$ (resp. $GL_n(\mathbb Q_p)$). 
\item the truncated traces $\mathrm{Tr}( C_{\lambda}f_{n\alpha s}, \rho_{\pi,p})$ where $\rho_{\pi,p}$ are representations of $GL_n(\mathbb Q_p)$ of type I or type II (as defined in Definition 5.3) and are computed in Proposition 5.3.

\end{itemize}
\end{theoremA}

We make some remarks on the notation appearing in the above theorem: 

\begin{itemize}
\item The function $f^{p\infty} \in \mathcal H(G(\mathbb A_f^p))$ is a $K^p$-bi-invariant Hecke operator. 
\item The function $f_{n \alpha s} \in \mathcal H^{\mathrm{unr}}(G(\mathbb Q_p))$ is a certain unramifed Hecke operator whose Satake transform is given by $p^{\frac{\alpha s(n-s)}{2}}\sum\limits_{\substack{I \subset \{1,2,\ldots n\}\\ |I| = s}}$, where $X^I:= \prod_{i \in I }X_i$. 
\item The function $C_{\lambda} \in \mathcal H(G(\mathbb Q_p))$ is a characteristic function associated with the isocrystal $b$. 
\item The space $\mathcal A(G)$ is the space of automorphic forms on $G(\mathbb A)$.
\item The representation $\zeta$ is an irreducible algebraic representation over $\overline{\mathbb Q}$ of $G_{\overline{\mathbb Q}}$, associated with $\ell$-adic local system on $\bar S$ (where $\ell \not = p$ is a prime number). Let $\overline{\mathbb Q}_{\ell}$ be an algebraic closure of $\mathbb Q_{\ell}$ with an embedding $\overline{\mathbb Q} \subset \overline{\mathbb Q}_{\ell}$. We write $\mathcal L$ for the $\ell$-adic local system on $\bar S$ associated to the representation $\zeta \otimes \mathbb Q_{\ell}$ of $G_{\bar{\mathbb Q}_{\ell}}$.
\end{itemize}

We now comment on the strategy for proofing Theorem A. In \cite{Kret11}, Kret truncates the formula of Kottwitz and adapts the argument in \cite{Kot92} (see Proposition 9 in \cite{Kret11}), which can be further implemented in our setting. Then
we have $$\mathrm{Tr}(\mathrm{Frob}^{\alpha}_q\times f^{p \infty},\sum\limits_{i=0}^{\infty}H^i_{et,c}(S_{b},\iota^*\mathcal L)) = |\mathrm{ker}^1(\mathbb Q,G)|\mathrm{Tr}(f^{p\infty } (X^{\alpha}\otimes C_{\lambda}f_{n\alpha s})f_{ \zeta, \infty}),$$ where function $f_{ \zeta, \infty}$ is an Euler-Poincar{\'e} function associated to $\zeta$ (as introduced in \cite{Clo85}). Also note that $$\mathrm{Tr}(f^{p\infty } (X^{\alpha}\otimes C_{\lambda}f_{n\alpha s})f_{ \zeta, \infty},\mathcal A(G)) = \sum\limits_{\pi \subset \mathcal A(G)} \mathrm{Tr}(X^{\alpha} \otimes C_{\lambda}f_{n\alpha s},\pi_p) \mathrm{Tr}(f^{p\infty},\pi^{p\infty})\mathrm{Tr}(f_{ \zeta, \infty},\pi_{\infty}),$$ and $$\mathrm{Tr}(f_{ \zeta, \infty},\pi_{\infty}) =N_{\infty}\sum_{i=1}^{\infty}(-1)^idim(H^i(\mathfrak g, K_{\infty};\pi_{\infty}\otimes \zeta))$$ where $N_{\infty}$ is a certain explicit constant as defined in $\S 3$ of \cite{Kot92}, we reduce the computation of the trace $\mathrm{Tr}(f^{p\infty } (X^{\alpha}\otimes C_{\lambda}f_{n\alpha s})f_{ \zeta, \infty})$ to the computation of the trace $\mathrm{Tr}(X^{\alpha} \otimes C_{\lambda}f_{n\alpha s},\pi_p)$. We write $\pi_p \cong \sigma_{\pi,p} \otimes \tau_{\pi,p}$, where $\sigma_{\pi,p}$ (resp. $\rho_{\pi,p}$) is a representation of $\mathbb Q_p^{\times}$ (resp. $GL_n(\mathbb Q_p)$). Then we only need to compute $\mathrm{Tr}( C_{\lambda}f_{n\alpha s},\tau_{\pi,p}).$.
Using the approach developed in $\S 1$ of \cite{Kret12b}, we reduce the computation of $\mathrm{Tr}( C_{\lambda}f_{n\alpha s},\tau_{\pi,p})$ to the computation of $\mathrm{Tr}( C^{GL_{p_1 } \times GL_{p_2}(F)}_{c}f_{n\alpha s},J_{N_{(p_1,p_2)}}(\tau_{\pi,p}))$, where $C^{GL_{p_1 } \times GL_{p_2}(F)}_{c}$ is the characteristic function of the set of compact elements in $(GL_{p_1 } \times GL_{p_2})(F)$ (the definition of compact elements is given in Definition 3.1), and $J_{N_{(p_1,p_2)}}(\tau_{\pi,p}))$ is the normalised Jacquet module of $(\tau_{\pi,p}))$ (the group $N_{(p_1,p_2)}$ is the unipotent part of the standard parabolic subgroup $P$ corresponding to partition $(p_1,p_2)$, as defined in $\S 3$). 

Let $\pi = \mathrm{Ind}_{P}^{GL_{p_1} \times GL_{p_2}}(\rho)$ be an induced representation of $(GL_{p_1}\times GL_{p_2})(F)$ (as defined in $\S 2$), where $P$ is a standard parabolic subgroup of $GL_{p_1} \times GL_{p_2}$, by repeating the proof Lemma 5 in \cite{Kret11} (see also Proposition 5.1), we have $\mathrm{Tr}( C^{GL_{p_1 } \times GL_{p_2}(F)}_{c}f_{n\alpha s},\pi) = 0$. Theorem 2 of \cite{Kret11} restricts the possible forms of $\tau_{\pi,p}$ (more precisely, the representation $\tau_{\pi,p}$ must be a semi-stable rigid representation, as defined in Definition 5.1, see also $\S 3$ of \cite{Kret11}). By Theorem VI.5.1 of \cite{Renard10}, we have an explicit description of the Jacquet module $J_{N_{(p_1,p_2)}}(\tau_{\pi,p})$ in terms of certain minimal elements in the symmetric group $S_n$ (as studied in $\S 4$). And we prove that all irreducible subquotients of $J_{N_{(p_1,p_2)}}(\tau_{\pi,p})$ are induced representations unless $\tau_{\pi,p}$ are of certain specific forms (see Proposition 5.2 and Definition 5.3), for which we can derive explicit formulas for the traces $\mathrm{Tr}( C^{GL_{p_1 } \times GL_{p_2}(F)}_{c}f_{n\alpha s},J_{N_{(p_1,p_2)}}(\tau_{\pi,p}))$ (see Proposition 5.3). 

\textbf{Structure of the paper:} In $\S3$ we carry out the necessary local computations at $p$: We introduce the Kottwitz functions $f_{n\alpha s}$ and certain truncated Kottwitz functions. Then, we study truncated traces of $\mathrm{Tr}(C_{\lambda} f_{n\alpha s}, \pi)$, and we reduce the computation of the truncated trace $\mathrm{Tr}(C_{\lambda} f_{n\alpha s}, \pi)$ to the computation of certain compact trace, more precisely, we prove that (see Proposition 3.4) $$\mathrm{Tr}(C_{\lambda} f_{n\alpha s}, \pi) =\mathrm{Tr}(C^{M_{\lambda}(F)}_{c}  (f_{n_1 \alpha s_1} \otimes \ldots \otimes f_{n_k \alpha s_k}), \pi_{N_{\lambda}}(\delta_{GL_n, P_{\lambda}}^{-\frac 12}) ),$$ where $\pi$ is an admissible representation of finite length. Note that we can compute compact traces using Proposition 2.1 of \cite{Clo90}. 

 In $\S 4$ we give a combinatorial interpretation of certain double cosets of symmetric groups in terms of Young tableaux (see Proposition 4.3). Due to Theorem VI.5.1 of \cite{Renard10}, the combinatorial structure of certain double cosets of symmetric groups plays a key role in the study of Jacquet modules of induced representations of general linear groups. The results established in $\S 4$ will be widely used in $\S 5$. 

 We continue the local computations in $\S5$: In this section, we study the traces of $C_{\lambda} f_{n\alpha s}$ against semi-stable rigid representations, which are local components of discrete automorphic representations of general linear groups. In Proposition 3.4, we have $$\mathrm{Tr}(C_{\lambda} f_{n\alpha s}, \pi) =\mathrm{Tr}(C^{GL_{p_1}\times GL_{p_2}(F)}_{c}  (f_{n_1 \alpha p_1} \otimes f_{p_2 \alpha s_2}), \pi_{N_{(p_1,p_2)}}(\delta_{GL_n, P_{(p_1,p_2)}}^{-\frac 12}) )$$ for certain $\lambda$. By making certain assumptions about $\lambda$ (as stated before Proposition 5.2), we prove that $\mathrm{Tr}(C^{GL_{p_1}\times GL_{p_2}(F)}_{c}  (f_{n_1 \alpha p_1} \otimes f_{p_2 \alpha s_2}), \pi_{N_{(p_1,p_2)}}(\delta_{GL_n, P_{(p_1,p_2)}}^{-\frac 12}) ) \not = 0$ unless $\pi$ is of certain specific forms (as defined in Definition 5.2), and we compute the traces $\mathrm{Tr}(C_{\lambda} f_{n\alpha s}, \pi)$ in Proposition 5.3.

 Having established the necessary local computation in $\S 5$, we now turn to the global setting in $\S6$, where we use the Langlands-Kottwitz method and automorphic forms to study the Frobenius-Hecke trace of the $\ell$-adic cohomology of certain Newton strata of Kottwitz varieties: Firstly, we introduce Kottwitz varieties and certain Newton strata of Kottwitz varieties. Then in the proof of Theorem 6.1, we express the Frobenius-Hecke trace of the $\ell$-adic cohomology of the strata in terms of automorphic representations of a certain unitary group associated to division algebra, and further we express the Frobenius-Hecke trace in terms of traces $\mathrm{Tr}( C_{\lambda}f_{n\alpha s}, \rho_{\pi,p})$, which are computed in Proposition 5.3. 
 
\section*{Acknowledgements}
This paper is a part of the author’s PhD research at the
University of Amsterdam, under the supervision of Arno Kret. I thank my supervisor Arno Kret for all his help and for his valuable comments on the preliminary drafts of the paper. I also thank Dhruva Kelkar for helpful discussions.

\section{Notations}

We introduce the following notation:

 \begin{itemize}

\item In the following, let $G$ be a quasi-split reductive group over $F$, together with a chosen Borel subgroup $P_0 $ and a Levi decomposition $P_0 = M_0 N_0$. Let $P = MN \subset G$ be a standard parabolic subgroup, the modular character $\delta_{G,P} (m)$ is defined to be $|\mathrm{det}(m,\mathfrak n)|$, where $m \in M(F)$ and $\mathfrak n$ is the Lie algebra of $N$ (we also use $\delta_{P}$ to denote $\delta_{G,P}$ if there is no ambiguity).

\item Let $\mathrm{Ind}^G_P$ denote the normalized parabolic induction (which sends unitarty representations to unitary representations, as defined in $\S$VI.1.2 in \cite{Renard10}, note that the modular character $\delta_{G,P} (m)$ is defined to be $|\mathrm{det}(m,\mathfrak n)|^{-1}$ in \cite{Renard10}), which is different from our convention (we also use $\delta_{G,P}$ to denote $\delta_{G,P}$ when there is no ambiguity). 

\item Let $\pi$ be a representation of $G(F)$. The Jacquet module $\pi_N$ of a representation $\pi$ is the space of coinvariants for the unipotent subgroup $N \subset G$. We use $J_N(\pi)$ to denote the normalized Jaquet module, which is $\pi_N(\delta_{G,P}^{-\frac 12})$.

\item We use $\mathrm{St_{G(F)}}$ to denote the Steinberg representation of $G(F)$ (see \cite{Rod86}), and we use $ \mathrm{1}_{G(F)}$ to denote the trivial representation of $G(F)$. 

\item We use $\mathcal H(G(F))$ to denote the Hecke algebra of locally constant, compactly supported functions on $G(F)$. Let $\mathcal H(G(F)//L)$ denote the Hecke algebra of locally constant, compactly supported functions on $G(F)$ which are bi-invariant under the action of $L$, where $L$ is a compact open subgroup of $G(F)$. 

\item In the following, let $G$ be unramified with a hyperspecial subgroup $K \subset G(F)$, we $\mathcal H^{\mathrm{unr}}(G(F))$ to denote $\mathcal H(G(F)//K)$, elements of $\mathcal H(G(F)//K)$ are called unramified Hecke operators. Let $P=MN \subset G$ be a parabolic subgroup, we use $K_M$ to denote $K \cap M(F)$, and we normalize the Haar measure on $K$ (resp.$K_M$) so that $\mu(K) =1$ (resp. $\mu(K_M) = 1$). Given $f \in \mathcal H(G(F))$, we define $\bar f \in \mathcal H(G(F))$ by $g \mapsto  \int_K f(kgk^{-1}) dg$.

\item Let $P = MN \in \mathcal P$.
For $f \in \mathcal H(G(F))$ and for $m \in M(F)$,
the constant terms $f^{(P)} $ is defined to be $$f^{(P)}(m) = \delta_{G,P}^{\frac 12}(m)\int_{N(F)} f(mn) dn.$$ Note that the constant terms $f \mapsto f^{(P)}$ is injective (see \cite{Kot80}, $\S$5).

  \item Let $P_0 = TU$ where $T \subset P_0$ is a maximal torus of $G$, we use $$\mathcal S_G: \mathcal H(G(F)//K) \rightarrow  \mathcal H(T(F)//T(F) \cap K)$$ to denote the Satake transform, where $\mathcal H(G(F)//K) \rightarrow  \mathcal H(T(F)//T(F) \cap K)$ is defined to be $f \mapsto f^{(P_0)}$. Note that $\mathcal S_G$ has image in $\mathcal H(T(F)//T(F) \cap K)^{W(G,T)(F)}$, where $W(G,T) = N_G(T)/T$ is the Weyl group and $N_G(T)$ is the normalizer of $T$ in $G$. We also use $\mathcal S$ to denote $\mathcal S_G$ if there is no ambiguity. 
  
  \item Let $S \subset  T$ be a maximal split torus, and let $d$ denote the dimension of $S$. The Hecke algebra $\mathcal H(T(F)//T(F) \cap K) $ is isomorphic to $ \mathcal H(S(F)//S(F) \cap K)$, with isomorphism given by $f  \mapsto f|_{S(F)}$. And we have $\mathcal H(S(F)//S(F) \cap K) \xrightarrow{\sim} \mathbb C[X_1^{\pm 1}, \ldots, X_d^{\pm 1}]$ as described in 2.2 of \cite{Minguez11}. 

  \item Given $n \in \mathbb Z_{>0}$, a partition of $n$ is a tuple $(n_1, \ldots, n_k) \in \mathbb Z^k_{>0}$ for some $k \in \mathbb Z_{>0}$. And an extended partitions of $n$ is a tuple $(n_1, \ldots, n_k) \in \mathbb Z^k_{\geq 0}$ for some $k \in \mathbb Z_{>0}$. And $k$ is called the length of the partition.

\end{itemize}

\section{Certain truncations of $f_{n\alpha s}$}

Let $F$ be a p-adic field and let $\mathcal O_F \subset F$ be the ring of integers. Let $\pi_F \in \mathcal O_F$ be a uniformizer. We write $q  = |\mathcal O_F/(\pi_F)| = p^r$, we normalize the $p$-adic norm on $F$ so that $|\pi_F| = p^{-r}$. Let $n \in \mathbb Z_{>0}$. We fix a Borel subgroup $P_{0,n} = M_{0,n}N_{0,n} \subset GL_n$ consisting of upper triangular matrices, where $M_{0,n}$ is the Levi part, and we use $\mathcal P_n$ to denote the set of standard parabolic subgroups of $GL_n$. Given a partition $(n_1, \ldots n_k)$ of $n $, we use $P_{(n_1, \ldots n_k)}$ (resp. $M_{(n_1, \ldots n_k)}$ and $N_{(n_1, \ldots n_k)}$) to denote the standard parabolic subgroup (resp. Levi subgroup and unipotent subgroup) corresponding to $(n_1, \ldots n_k)$. We call $k$ the length of the partition $(n_1, \ldots, n_k)$, respectively, the length of the parabolic subgroup $P_{(n_1, \ldots, n_k)}$. 
Let $F \subset \bar F$ be an algebraic closure, we embed $GL_n(F)$ into $GL_n(\bar F)$.

The unramifed Hecke algebra $\mathcal H^{\mathrm{unr}}(GL_n(F))$ is isomorphic to $$\mathbb C[X_1^{\pm1}, \ldots, X_n^{\pm1}]^{S_n} \hookrightarrow \mathbb C[X_1^{\pm1}, \ldots, X_n^{\pm1}] \cong \mathcal H^{\mathrm{unr}}(M_{0,n}(F))$$ such that $X_k$ corresponds to the characteristic function of $M_{0,n}(\mathcal O_F) a_k  M_{0,n}(\mathcal O_F)$, where $a_k = \mathrm{diag}(a_{k,i}) \in T(F)$ such that $a_{k,k} = \pi_F $ and $a_{k,i} = 1 $ if $i \not = k$. And the action of $\sigma \in S_n$ on $\mathbb C[X_1^{\pm1}, \ldots, X_n^{\pm1}]$ is given by $\sigma X_i = X_{\sigma(i)}$. When $n = 1$, we also write $\mathcal H^{\mathrm{unr}}(F^{\times}) \cong \mathbb C[Y]^{\pm1}$, such that $Y$ corresponds to the characteristic function of $ \pi_F  \mathcal O_F^{\times}$. 

We define $f_{n\alpha s} \in \mathcal H^{\mathrm{unr}}(GL_n(F)) $ as the unramified Hecke operator whose Satake transform is given by $$q^{\frac{\alpha s(n-s)}2}\sum\limits_{\substack{I \subset {1,2, \ldots, n}\\|I| = s}} (X^I)^ {\alpha}$$ where $X^I : = \prod_{i \in I}X_{i}$. For the constant terms of $f_{n\alpha s}$, we have the following proposition (Proposition 4 of \cite{Kret11}): 

\begin{proposition}
Let $P = M N \subset GL_n$ be a standard parabolic subgroup corresponding to the partition $(n_1, \ldots, n_k)$ of $n$. The constant term $f_{n\alpha s}^{(P)}$ is then given by $$\sum\limits_{(s_i)} q^{\alpha C(n_i,s_i)} f_{n_1\alpha s_1} \otimes  \cdots \otimes  f_{n_k\alpha s_k},$$ where the sum ranges over extended partitions $(s_i)$ of $s$ of length $k$, and the constant $C(n_i, s_i)$ is equal to $\frac{s(n-s)}2-\sum_{i = 1}^k \frac{s_i(n_i-s_i)}2$.
\end{proposition}

Let $(n_{1}, \ldots, n_k)$ be a partition of $n$, let $(n_{i,1}, \ldots,n_{i, l_i})$ be a partition of $(n_i)$ for every $1 \leq i \leq k$. Then $(n_{i,j})$ (where $1 \leq i \leq k, 1 \leq j \leq l_i$, ordered in lexicographic order) is a partition of $n$. Let $P_{1} = M_1 N_1$ be the parabolic subgroup of $GL_n$ corresponding to $(n_{1}, \ldots, n_k)$, and let $P_2 = M_2 N_2 \subset P_1$ be the parabolic subgroup corresponding to $(n_{i,j})$ (where $1 \leq i \leq k, 1 \leq j \leq k_i$, we also use $(m_1, \ldots m_{l_1 + \cdots l_k})$ to denote the partitions corresponding to $P_2$). We now introduce the following truncations of $f_{n\alpha s}$:

Let $\chi^{M_1}_ {M_2}, \hat\chi^{M_1}_{M_2} \in \mathcal H^{\mathrm{unr}}(M_2(F))$ be the functions introduced in pp.492-493 of \cite{Kret11}, then we have the following proposition:

\begin{proposition}

\begin{enumerate}

    \item  The function $\chi^{M_1}_ {M_2}f_{n\alpha s}^{(P_2)}\in\mathcal H^{\mathrm{unr}}(M_2(F))$ equals $$\sum\limits_{(s_i)} q^{\alpha C(m_i,s_i)} \otimes_{j = 1}^{l_1 + \cdots l_k} \phi_{m_i\alpha s_{i}}, $$ where the sum ranges over extended partitions $(s_i)$ of $s$ of length $l_1 + \cdots + l_k$, where the partitions satisfy

$$s_{l_i+ 1}/m_{l_i + 1} >s_{l_i+2}/m_{l_i+2} > \cdots > s_{l_{i+1}}/m_{l_{i+1}} $$
 for all indices $i \in \{0,1,2,\ldots,k-1\}$ (we define $l_0$ to be $0$).

\item Let $(s_1, \ldots, s_k)$ be an extended partition of $s$. Then the function $\hat{\chi}^{M_1}_ {M_2}(\otimes_{i=1}^kf_{n_i\alpha s_i})^{(P_2)}\in\mathcal H^{\mathrm{unr}}(M_2(F))$ equals $$\sum\limits_{(s_{1,j}), \ldots, (s_{k,j})} q^{\alpha \sum\limits_{i = 1}^k C(n_{i,j},s_{i,j})} \otimes_{i = 1}^{k}(\otimes_{j=1}^{l_i} \phi_{n_{i,j}\alpha s_{i,j}}), $$ where the sum ranges over $(s_{1,j}), \cdots (s_{k,j})$, such that $(s_{i,j})$ is an extended partition of $s_i$ of length $l_i$ for every $1 \leq i \leq k $, satisfying

$$  \sum\limits_{j = 1}^{b_i} s_{i,j}< \frac {s_i}{n_i}  \sum\limits_{j = 1}^{b_i} n_{i,j}$$
 for $i \in \{1,2,\ldots,k\}, b_i \in \{1,2 \ldots,l_i - 1 \}$.

\end{enumerate}
\end{proposition}

\begin{proof}
This follows by repeating the proof of Proposition 5 in \cite{Kret11}.
\end{proof}

\begin{definition}
Let $M = GL_{n_1} \times \cdots GL_{n_k} \subset GL_n$ be a Levi subgroup. We define $\Omega_{M,c} \subset M(F)$ as the set consisting of $g \in M(F) = GL_{n_1}(F) \times \cdots GL_{n_k}(F)  \subset GL_n(\bar F)$, such that the $p$-adic norms of the eigenvalues of $GL_{n_i}$ are the same for $1 \leq i \leq k$ (elements of $\Omega_{M,c}$ are called compact elements of $M(F)$). We use $C^{M(F)}_{c}$ to denote the characteristic function of $\Omega_{M,c}$.
\end{definition}

Note that the above definition is a special case of compact elements (in a $p$-adic group) as given in $\S 2.2 $ of \cite{Kret11}.
Let $\mathbb R^n_{\geq} \subset \mathbb R^n$ denote $\{(r_1, \ldots, r_n) \in \mathbb R_n: r_1 \geq \cdots \geq r_n\}$. Given $g \in GL_n(\bar F)$, let $\Phi(g) = (t_{1,g}, \ldots, t_{n,g})\in \mathbb R^n_{\geq}$ denote the n-tuple of the inverses of $p$-adic valuations of the eigenvalues of $g$. For example, for $g = \mathrm{diag}(p^{-2}, p^{-1}, p^{-3}) \in GL_3(\mathbb Q_p)$, we have $\Phi(g ) = (3,2,1)$. 

\begin{definition}
 Given $b \in \mathbb R^n_{\geq}$, let $\Omega_{b} \subset Gl_n(\bar F)$ denote the set consisting of $g \in GL_n(\bar F)$ such that $\Phi(g) = \lambda b$ for some $\lambda \in \mathbb R_{>0}$. We use $C_{b}$ to denote the characteristic function of $\Omega_b$.
\end{definition}

The above definition is a special case of the truncation given in Definition 1.3 of \cite{Kret12b}.
Let $s \in \mathbb Z_{>0}$ such that $s\leq n$. Given a partition $(n_1, \ldots n_k)$ of $n$, we choose a $\lambda = (\lambda_1^{ n_1},\ldots ,\lambda_k^{n_k}) \in \mathbb R^n_{\geq}$ such that the following conditions hold:
\begin{itemize}
\item $\lambda_i n_i \in \mathbb Z_{\geq 0} $ for $1\leq i \leq k$. 
\item $\sum_{i = 1}^k \lambda_i n_i = s$. 
\item $\lambda_i > \lambda_j$ for all $1 \leq i < j \leq k$. 

\end{itemize}
Note that the vector $\lambda = (\lambda_1^{ n_1},\ldots ,\lambda_k^{n_k})$ is the Newton vector of certain isocrystals of rank $n$ (see $\S 1$ of \cite{Ra02}). 
Let $P_{\lambda} = M_{\lambda}N_{\lambda} \subset GL_n$ denote the standard parabolic subgroup corresponding to partition $(n_1, \ldots, n_k)$. Let $s_i$ denote $n_{i}\lambda_i$ (by assumption, we have $s_i \in \mathbb Z_{\geq 0}$). Then we have the following proposition:

\begin{proposition}
\begin{enumerate}
\item Let $P = MN \subset GL_n$ be a standard parabolic subgroup such that $P \not = P_{\lambda} $, then the function $C^{M(F)}_{c}C_{\lambda} \chi^{GL_n}_{M} f_{n\alpha s}^{(P)}  \in \mathcal H(M_{\lambda}(F))$ equals $0$.
    \item The function
$C^{M_{\lambda}(F)}_{c}C_{\lambda} \chi^{GL_n}_{M_\lambda} f_{n\alpha s}^{(P_{\lambda})}  \in \mathcal H(M_{\lambda}(F))$ equals $q^{\alpha C(n_i,s_i)}C^{M_\lambda(F)}_c(f_{n_1 \alpha s_1} \otimes \ldots \otimes f_{n_k \alpha s_k})$.
\end{enumerate}

\end{proposition}

\begin{proof}
Let $P = MN \subset GL_n$ be a standard parabolic subgroup corresponding to partition $(m_1,\ldots, m_l)$. Applying Proposition 3.1, we have: $$f_{n\alpha s}^{(P)}= \sum\limits_{(s'_i)} q^{\alpha C(n_i,s_i)} f_{n_1\alpha s'_1} \otimes  \cdots \otimes  f_{n_k\alpha s'_k},$$ where the sum ranges over extended partitions $(s'_i)$ of $s$ of length $l$, and the constant $C(n_i, s'_i)$ is equal to $s(n-s)-\sum_{i = 1}^k s'_i(n_i-s'_i)$. Let $(s'_1, \ldots, s'_k)$ be an extended partition of $s$ of length $l$, and
let $(g_1, \ldots g_l) \in M(F)$ be an element in the support of
$C^{M(F)}_{c}C_{\lambda} \chi^{GL_n}_{M} (f_{n_1\alpha s'_1} \otimes  \cdots \otimes  f_{n_k\alpha s'_k}) \in \mathcal H(M(F))$. For $1 \leq i \leq l$, let $a_i$ denote the $p$-adic valuation of $\mathrm{det}(g_i)$, by applying Lemma 2 of \cite{Kret11}, we have $a_i = - \alpha s'_i$. 
By the definition of $C^{M(F)}_{c} $, the p-adic valuations of the eigenvalues of $g_i$ are equal to $-\frac{a_i}{m_i} =  -\frac{\alpha_is'_i}{m_i}$ for $1 \leq i \leq l$. By the definition of $\chi^{GL_n}_{M}$, we have $s'_1/m_1>s'_2/m_2 > \cdots > s'_l/m_l$. Then we have $$\Phi((m_1,\ldots, m_l)) = \alpha((\frac{s'_1}{m_1})^{m_1}, \ldots , (\frac{s'_l}{m_l})^{m_l}) \in \mathbb R^n_{\geq}$$ (with $\Phi$ as defined before Definition 3.2). By the definition of $C_{\lambda}$, we have $$\Phi((m_1, \dots , m_l)) = \lambda (\lambda_1^{ n_1},\ldots ,\lambda_k^{n_k})$$ for some $\lambda \in \mathbb R_{>0}$, then we have $$\alpha((\frac{s'_1}{m_1})^{m_1}, \ldots , (\frac{s'_l}{m_l})^{m_l}) = \lambda (\lambda_1^{ n_1},\ldots ,\lambda_k^{n_k}),$$ By comparing terms, we must have $(m_1, \ldots , m_l) = (n_1, \ldots, n_k)$. Therefore, the function $C^{M(F)}_{c}C_{\lambda} \chi^{GL_n}_{M} f_{n\alpha s}^{(P)}  \in \mathcal H(M_{\lambda}(F))$ equals $0$ if $P \not = P_{\lambda}$, and $$C^{M_\lambda(F)}_cC_{\lambda} \chi^{GL_n}_{M_\lambda}(f_{n_1 \alpha s'_1} \otimes \ldots \otimes f_{n_k \alpha s'_k}) = 0$$ if $(s'_1, \ldots,s'_k) \not = (s_1, \ldots, s_k)$. Also note that $$\chi^{GL_n}_{M_\lambda}(f_{n_1 \alpha s_1}\otimes \ldots \otimes f_{n_k \alpha s_k}) = f_{n_1 \alpha s_1}\otimes \ldots \otimes f_{n_k \alpha s_k}$$ (see Proposition 5 of \cite{Kret11}), we have $$C^{M_{\lambda}(F)}_{c}C_{\lambda} \chi^{GL_n}_{M_\lambda} f_{n\alpha s}^{(P_{\lambda})}  =C^{M_\lambda(F)}_cC_{\lambda} \chi^{GL_n}_{M_\lambda}(f_{n_1 \alpha s_1} \otimes \ldots \otimes f_{n_k \alpha s_k}) = C^{M_\lambda(F)}_c(f_{n_1 \alpha s_1} \otimes \ldots \otimes f_{n_k \alpha s_k}).$$
\end{proof}

Now we can prove the main result of this section, which will be used extensively in $\S5$ to compute truncated traces against semi-stable rigid representations. 

\begin{corollary}
Let $\pi$ be an admissible representation of $GL_n(F)$ of finite length, then we have $$\mathrm{Tr}(C_{\lambda} f_{n\alpha s}, \pi) =\mathrm{Tr}(C^{M_{\lambda}(F)}_{c}  (f_{n_1 \alpha s_1} \otimes \ldots \otimes f_{n_k \alpha s_k}), \pi_{N_{\lambda}}(\delta_{GL_n, P_{\lambda}}^{-\frac 12}) )$$
\end{corollary}

\begin{proof}
By Proposition 2.1 of \cite{Clo90}, we have $$\mathrm{Tr}(C_{\lambda} f_{n\alpha s}, \pi) =\sum\limits_{P = MN \in \mathcal P} \mathrm{Tr}(C^{M(F)}_{c}\chi^{GL_n}_{M}  \overline{C_{\lambda}f_{n\alpha s}}^{(P)}, \pi_{N}(\delta_{GL_n, P}^{-\frac 12}) ).$$ Note that the function $C_{\lambda}$ is stable under conjugation, and that the function $f_{n\alpha s}$ is $K$-bi-invariant, we have $\overline{C_{\lambda}f_{n\alpha s}}^{(P)} = C_{\lambda}f_{n\alpha s}^{(P)}$. Then the proposition follows from Proposition 3.3. 
\end{proof}

Let $I_n$ denote the subgroup of $GL_{n}(\mathcal O_F) \subset GL_n(F)$ such that $g \in I_n$ if and only if $g$ reduces an element of the group $P_{0,n}(\mathcal O_F/(\pi_F))$ modulo $\pi_F$, the group $I_n$ is called the standard Iwahori subgroup of $Gl_n(F)$. Then we have the following proposition:

\begin{proposition}
Let $\pi$ be an admissible representation of $GL_n(F)$ of finite length such that $\mathrm{Tr}(C_{\lambda} f_{n\alpha s}, \pi) \not =0 $, then we have $\pi^{I_n} \not = \{0\}$.
\end{proposition}

\begin{proof} (cf. Corollary 1 in \cite{Kret11})
By Proposition 3.4, we have $$\mathrm{Tr}(C_{\lambda} f_{n\alpha s}, \pi) =\mathrm{Tr}(C^{M_{\lambda}(F)}_{c}  (f_{n_1 \alpha s_1} \otimes \ldots \otimes f_{n_k \alpha s_k}), \pi_{N_{\lambda}}(\delta_{GL_n, P_{\lambda}}^{-\frac 12}) ). $$ By the Corollary in p259 of \cite{Clo90}, there exists a standard parabolic subgroup $P = MN \subset P_{\lambda} \subset GL_n$ such that $\pi_N$ is unramified (that is, the space $\pi^{GL_n(\mathcal O_F)\cap M(F)} \not = \{0\}$). By Proposition 2.4 of \cite{Casselman801}, we can conclude that $\pi^{I_n} \not = \{0\}$. 
\end{proof}

\section{Double cosets of symmetric groups}

Let $S_n$ denote the permutation group on $\{1,2,\ldots,n\}$, and we embed $S_n$ into $GL_n(F)$ such that $ w \in S_n\mapsto (a_{i,j}) \in GL_n(F)$ where $(a_{i,j}) = 1 $ if $w(i) = j$, and $a_{i,j} = 0$ otherwise. Given $w \in S_n$ and $g \in GL_n(F)$, we define $w\cdot g$ to be $w^{-1} g w$. 

The group $S_n$ is generated by transposes $(1,2), \ldots, (n-1, n)$ (which we denote by $t_1, \ldots, t_{n-1}$). For $w \in S_n$, we define the length of $w$ (denoted by $\ell (w)$) to be $\min\{k: w = w_1w_2\ldots w_k: w_i \in \{t_1, \ldots, t_{n-1}\}\} $. For $i, j \in \{1,2,\ldots,n\}$, we say that $(i,j)$ is an inversion of $w$ if $i < j $ and $w(i) > w(j)$. Let $inv(w)$ to denote the total number of inversions of $w$, then we have the following proposition (for the proof, see Lemma 2.1 of \cite{Wildon}):

\begin{proposition}
For every $w \in S_n$, we have $inv(w) = \ell(w)$. 
\end{proposition}

Let $ (n_1, \ldots, n_k)$ be a partition of $n$, we use $S_{(n_1, \ldots, n_k)} \subset S_n$ to denote $S_{n_1} \times \ldots \times S_{n_k}$. Let $P$ be a standard parabolic subgroup of $GL_n$, we use $S_P$ to denote the subgroup of $S_n$ associated to the partition of $n$ corresponding to $P$. Let $\lambda$ and $\nu$ be two partitions of $n$, we have the following proposition:

\begin{proposition}
For an arbitrary double cosets $S_{\lambda}g S_{\nu} \subset S_n$, there exists a unique $g_{min} \in S_{\lambda}g S_{\nu}$ such that $\ell(g_{min})$ is minimal among elements in $S_{\lambda}g S_{\nu}$. 
\end{proposition}
\begin{proof}
See V.4.6 of \cite{Renard10}.
\end{proof}

We call $g_{min} \in S_{\lambda}g S_{\nu}$ the \textbf{minimal length representative of $S_{\lambda}g S_{\nu}$}, and we use $S^{min}_{\lambda, \nu}$ to denote the set of minimal length representatives of $S_{\lambda}\setminus S_n /S_{\nu}$.

Let $ (n_1, \ldots, n_k)$ be a partition of $n$. We define the \textbf{diagram of $(n_1, \ldots, n_k)$} as the set $\{(i, j) : 1 \leq i \leq k, 1 \leq j \leq n_i\}$, which we denote by $T^{(n_1, \ldots, n_k)}$.
We represent the \textbf{diagram of $(n_1, \ldots, n_k)$} by a Young diagram in the usual way. For example $T^{(5,2,3)}$ is represented by
$$
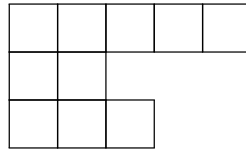
\begin{ytableau}
\empty &\empty &\empty&\empty &\empty \\
\empty &\empty  \\
\empty & \empty &\empty
\end{ytableau}$$ 
We label the tableau $T^{(n_1, \ldots, n_k)}$ as follows: We map $(i,j)$ to $\sum_{l = 1}^{n_{i-1}} n_l + j$ (we define $n_0 = 0$), for example, the tableau $T^{(5,2,3)}$ is labeled as follows: 
$$
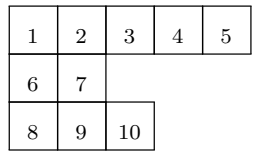
\begin{ytableau}
_1 &_2 &_3&_4&_5 \\
_6 &_7  \\
_8 & _9 &_{10}
\end{ytableau}$$

Let $(n_1, \ldots, n_k),  (m_1, \ldots, m_s)$ be two partitions of $n$. We define \textbf{$(n_1, \ldots, n_k)$-tableaux of type $(m_1, \ldots, m_s)$} to be functions from elements of $T^{(n_1, \ldots, n_k)}$ to $\mathbb Z_{>0}$, which takes $i \in \{1, \ldots , s\}$ exactly $m_i$ times (following the order of label), which we denote by $T^{(n_1, \ldots, n_k)}_{(m_1,\ldots,m_s)}$. The tableau $T^{(n_1, \ldots, n_k)}_{(m_1,\ldots,m_s)}$ is represented as follows:
We draw the diagram $T^{(n_1, \ldots, n_k)}$, then we write the image of each element inside the corresponding box. For example, the tableau $T_{(3,3,4)}^{(5,2,3)}$ is represented by
$$\begin{ytableau}
1 &1  &1 & 2&2 \\
2 &3  \\
3 & 3 &3
\end{ytableau}$$

We regard the tableau $T^{(n_1, \ldots, n_k)}_{(m_1,\ldots, m_s)}$ as a function from $\{1,2,\ldots,n\}$ to $\{1,\ldots,s\}$. 
We define the action of $S_n$ on $T^{(n_1, \ldots, n_k)}_{(m_1,\ldots, m_s)}$ by "permutation of labels", that is, given a $w \in S_n$, we define $w T^{(n_1, \ldots, n_k)}_{(m_1,\ldots, m_s)}$ by $$w T^{(n_1, \ldots, n_k)}_{(m_1,\ldots, m_s)} (k) = T^{(n_1, \ldots, n_k)}_{(m_1,\ldots, m_s)} (w^{-1}k )$$ For example, $(2,3,4)(6,7) T_{(5,2,3)}^{(3,3,4)}$ is represented by 
$$\begin{ytableau}
1 &2 &1&1&2 \\
3&2  \\
3 & 3 &3
\end{ytableau}$$ 
We define the set $T^{min}_{\lambda,\nu}$ to be $$\{w T^{\lambda}_{\nu}: w \in S_n \text{ and }  wT^{\lambda}_{\nu} \text{ has weakly increasing rows; i.e., }T^{\lambda}_{\nu}(i, j ) \leq w T^{\lambda}_{\nu}(i , j+1)  \text{ for all } i,j.\}$$ Elements of the set $T^{min}_{\lambda,\nu}$ are called \textbf{row semi-standard $(n_1, \ldots, n_k)$-tableaux of type $(m_1, \ldots, m_s)$}.

\begin{proposition} Let $w$ be an element of $S_n$, and let $\lambda = (n_1, \ldots n_k), \nu = (m_1 \ldots m_s)$ be two partitions of $n$. Then $w$ is a minimal length representative in $S_{\lambda}\setminus S_n/ S_{\nu}$ if and only if:
\begin{itemize}
    \item $w T^{(n_1, \ldots, n_k)}_{(m_1,\ldots, m_s)}$ has weakly increasing rows; i.e., $w T^{(n_1, \ldots, n_k)}_{(m_1,\ldots, m_s)}(i, j ) \leq w T^{(n_1, \ldots, n_k)}_{(m_1,\ldots, m_s)}(i , j+1)  $ for all $i,j$, 
    \item if $i < j$ and the $ith$ and $jth$ positions of $w T^{(n_1, \ldots, n_k)}_{(m_1,\ldots, m_s)}$ are equal, then $w(i) < w(j)$.
\end{itemize}

And the map $w \mapsto w T^{(n_1, \ldots, n_k)}_{(m_1,\ldots, m_s)}$ gives a bijection between the set of minimal length representatives of $S_{\lambda}\setminus S_n /S_{\nu}$ (denoted by $S^{min}_{\lambda,\nu}$), and the set of row semi-standard $\lambda$-tableaux of type $\nu$ (denoted by $T_{\lambda,\nu}^{min}$). 
\end{proposition} 

\begin{proof}
See Theorem 4.1 and Corollary 5.1 of \cite{Wildon}.
\end{proof}

Let $\lambda = (n_1, \ldots,n_s ), \mu = (m_1, \ldots m_t)$ be two partitions of $n$. Let $A$ be a finite set, we use $S_A$ to denote the permutation group of elements of $A$. When $A = \emptyset$, we define $S_A$ to be the trivial group with one element. We define $\Lambda_j= \sum_{l = 1}^{j} n_l $ and $\Delta_j= \sum_{l = 1}^{j} m_l $ (we write $\Lambda_0  = 0$ and $\Delta_0 = 0$), then $$S_{\lambda} \cong S_{\{1,\ldots,\Lambda_1\}} \times S_{\{\Lambda_1 +1,\ldots,\Lambda_2\}} \times \ldots\times S_{\{\Lambda_{s-1} +1,\ldots,\Lambda_s\}}$$ and we have $$w S_{\lambda} w^{-1} = S_{\{w(1),\ldots,w(\Lambda_1)\}} \times S_{\{w(\Lambda_1 +1),\ldots,w(\Lambda_2\})} \times \ldots\times S_{\{w(\Lambda_{s-1} +1),\ldots,w(\Lambda_s)\}}$$
and $$S_{\nu} \cap wS_{\lambda} w^{-1} = \prod\limits_{1 \leq i \leq s, 1 \leq j \leq t} S_{\{w(\Lambda_{i-1} +1),\ldots,w(\Lambda_{i})\}} \cap S_{\{\Delta_{j-1} +1,\ldots,\Delta_{j}\}} $$ $$\cong \prod\limits_{1 \leq i \leq s, 1 \leq j \leq t} S_{\{w(\Lambda_{i-1} +1),\ldots,w(\Lambda_{i})\}\cap \{\Delta_{j-1} +1,\ldots,\Delta_{j}\}}.$$ 

When $w \in S_{\nu, \lambda}^{min}$, by Proposition 4.3, we have that the set $$\{w(\Lambda_{i-1} +1),\ldots,w(\Lambda_{i})\}\cap \{\Delta_{j-1} +1,\ldots,\Delta_{j}\}$$ consists of consecutive integers for $1 \leq i \leq s, 1 \leq j \leq t$. Then $P_{\nu} \cap wP_{\lambda}w^{-1}$ is a standard parabolic subgroup, with the corresponding Weyl group equals $$\prod\limits_{1 \leq i \leq s, 1 \leq j \leq t} S_{\{w(\Lambda_{i-1} +1),\ldots,w(\Lambda_{i})\}\cap \{\Delta_{j-1} +1,\ldots,\Delta_{j}\}}.$$ By the construction of the tableau $wT^{\nu}_{\lambda}$, we have the following proposition:
\begin{proposition}
The length of the parabolic subgroup $P_{\nu} \cap wP_{\lambda}w^{-1}$ (as defined at the beginning of $\S3 $) equals to $\sum_{i = 1}^t N(wT^{\nu}_{\lambda},i)$, where $N(wT^{\nu}_{\lambda},i)$ is the number of different entries in the $i$-th row. For example, for the tableau $N((2,3,4)(6,7) T_{(5,2,3)}^{(3,3,4)},1) = 2$ defined before Proposition 4.3, we have $N((2,3,4)(6,7) T_{(5,2,3)}^{(3,3,4)},1) = 2, N((2,3,4)(6,7) T_{(5,2,3)}^{(3,3,4)},2) = 2$ and $N((2,3,4)(6,7) T_{(5,2,3)}^{(3,3,4)},3) = 1$. 
\end{proposition}

\section{Certain representations of $GL_n(F)$}

Let $x, y, n \in \mathbb{Z}_{>0} $ be such that $n = xy$. We use $\mathrm{Speh}(x, y)$ to denote the Langlands quotient of the representation $$\mathrm{Ind}_{P_{(x\ldots x)}}^{GL_n}(| \mathrm{det} |^{\frac{y-1}2} \mathrm{St}_{GL_x(F)} \times | \mathrm{det} |^{\frac{y-3}2} \mathrm{St}_{GL_x(F)} 
 \times \ldots \times | \mathrm{det} |^{\frac{1-y}2} \mathrm{St}_{GL_x(F)}). $$ 
 
 \begin{definition}
 A representation $\pi$ of $GL_n(F)$ is called semi-stable rigid if it is unitarizable and is isomorphic to a representation of the form $\mathrm{Ind}_P^{GL_n} (\otimes_{i = 1}^k \mathrm{Speh}(x_i, y)(\varepsilon_i| \cdot |^{e_i}))$, where $P = MN$ is the parabolic subgroup corresponding to the composition $(yx_i)$ of $n$ and where the tensor product is taken along the blocks $M = \prod_{a = 1}^k GL_{yx_i}$ and
 
\begin{itemize}
\item $k \in \{1, 2, \ldots, n\}$; 
\item $\varepsilon_i$ is a unitary unramified character $\varepsilon_i: GL_x(F) \rightarrow \mathbb C^{\times}$ (thus $\varepsilon_i$ is of the form $g \mapsto |\mathrm{det}(g)|^{\mathbf id_i }$ for some $d_i \in \mathbb R$ ) for each $i \in \{1, 2, \ldots, k\}$;
\item $e_i$ is a real number $e_i$ in the open interval $(-\frac 12, \frac 12)$ for each $i \in \{1, 2, \ldots, k\}$;
\item positive integers $y, x_1, x_2, \ldots, x_k $ such that $\frac ny = \sum_{ i = 1 }^k x_i$. 
\end{itemize}
\end{definition}

We note that semi-stable rigid representations are irreducible (see Theorem 7.5 of \cite{Tad86}), therefore, different orderings of $x_1, \ldots x_k$ determine the semi-stable rigid representation $\mathrm{Ind}_P^{GL_n} (\otimes_{i = 1}^k \mathrm{Speh}(x_i, y)(\varepsilon_i| \cdot |^{e_i}))$ up to isomorphism). 

\begin{remark}
Let $\pi = \mathrm{Ind}_P^{GL_n} (\otimes_{i = 1}^k \mathrm{Speh}(x_i, y)(\varepsilon_i| \cdot |^{e_i}))$ be a semi-stable rigid representation of $GL_n(F)$, if $k =y =1$, then $\pi \cong \mathrm{1}_{GL_n(F)}(\varepsilon|\cdot|^e)$, where $\varepsilon$ is a unitary unramified character, and $e$ is a real number in the open interval $(-\frac 12, \frac 12)$. By applying Theorem 7.5 of \cite{Tad86}, we have $e =0$, therefore, the representation $\pi$ is isomorphic to $\mathrm{1}_{GL_n(F)}(\varepsilon)$ where $\varepsilon$ is a unitary unramified character. 

Similarly, if $k=x=1$, the representation $\pi$ is isomorphic to $\mathrm{St}_{GL_n(F)}(\varepsilon)$, where $\varepsilon$ is a unitary unramified character. 
\end{remark}

\begin{definition}
Let $(p_1,p_2)$ be a partition of $n$, and let $\pi$ be an admissible representation of $GL_n(F)$. We say that $\pi$ is of $(p_1,p_2)$-type if $\pi$ is isomorphic to $\mathrm{Ind}_P^{GL_n}(\mathrm{St}_{GL_{p_1}(F)} (\varepsilon_i| \cdot |^{e_i}) \otimes \mathrm{St}_{GL_{p_2}(F)} (\varepsilon_i| \cdot |^{e_i}))$ or $\mathrm{Ind}_P^{GL_n}(\mathrm{1}_{GL_{p_1}(F)} (\varepsilon_i| \cdot |^{e_i}) \otimes \mathrm{1}_{GL_{p_2}(F)} (\varepsilon_i| \cdot |^{e_i}))$, where $\varepsilon_1, \varepsilon_2$ are unitary unramified characters, and $e_1, e_2$ are real numbers in the open interval $(-\frac 12, \frac 12)$. 

\end{definition}

\begin{remark}

\begin{itemize}

\item Let $\pi = \mathrm{Ind}_P^{GL_n}(\mathrm{St}_{GL_{p_1}(F)} (\varepsilon_i| \cdot |^{e_i}) \otimes \mathrm{St}_{GL_{p_2}(F)} (\varepsilon_i| \cdot |^{e_i}))$ (respectively, $ \pi = \mathrm{Ind}_P^{GL_n}(\mathrm{1}_{GL_{p_1}(F)} (\varepsilon_i| \cdot |^{e_i}) \otimes \mathrm{1}_{GL_{p_2}(F)} (\varepsilon_i| \cdot |^{e_i}))$) be a representation of $(p_1, p_2)$-type, where $p_1 \not = p_2$, by Theorem 7.5 of \cite{Tad86}, we must have $e_1 = e_2 = 0$. 

\item If $p_1 = p_2$, let $\pi = \mathrm{Ind}_P^{GL_n}(\mathrm{St}_{GL_{p_1}(F)} (\varepsilon_i| \cdot |^{e_i}) \otimes \mathrm{St}_{GL_{p_2}(F)} (\varepsilon_i| \cdot |^{e_i}))$ (respectively, $ \pi = \mathrm{Ind}_P^{GL_n}(\mathrm{1}_{GL_{p_1}(F)} (\varepsilon_i| \cdot |^{e_i}) \otimes \mathrm{1}_{GL_{p_2}(F)} (\varepsilon_i| \cdot |^{e_i}))$), by Theorem 7.5 of \cite{Tad86}, we must have either $e_1 = e_2 = 0$, or $e_1 + e_2 = 0$ and $\varepsilon_1 = \varepsilon_2$. 

\end{itemize}

\end{remark}

The following proposition plays an important role in the computation of some compact traces: 

\begin{proposition}
Let $n, s \in \mathbb Z_{>0}$ such that $s \leq n$ and $(s,n) = 1$, and let $P = MN$ be a standard parabolic subgroup of $GL_n$, then we have $C^{GL_n(F)}_c f_{n\alpha s}^{(P)} = 0$. 
\end{proposition}

\begin{proof}
(cf. Lemma 5 of \cite{Kret11}) Let $(n_1, \dots, n_k)$ denote the partition corresponding to the parabolic subgroup $P$, by Proposition 4 of \cite{Kret11}, we have $$C^{GL_n(F)}_c f_{n\alpha s}^{(P)}=\sum\limits_{(s_i)} q^{\alpha C(n_i,s_i)} C^{GL_{n_1}(F)}f_{n_1\alpha s_1} \otimes  \cdots \otimes  C^{GL_{n_k}(F)}f_{n_k\alpha s_k},$$ where the sum ranges over certain extended partitions $(s_i)$ of $s$ of length $k$ (with the precise conditions given in Proposition 4 of \cite{Kret11}). Let $(m_1, \ldots, m_k ) \in M(F)$ be in the support of $C^{GL_n(F)}_c f_{n\alpha s}^{(P)}$, and let $\alpha_{i,1}, \ldots \alpha_{i,n_i}$ denote the eigenvalues of $m_i$. By the definition of $C_c^{GL_n(F)}$ and Lemma 2 of \cite{Kret11}, we have $|\alpha_{i,j}| = q^{-\frac{\alpha s_i}{n_i}} = q^{-\frac{\alpha s}{n}}$. Since $\frac sn = \frac {s_i}{n_i}$, it follows that $s_i = \frac{n_i s}{n}$. Also note that $ (n,s) = 1$, so we must have $n_i = n$, then we have our proposition. 
\end{proof}

\begin{lemma}
Let $f \in \mathcal H^{\mathrm{unr}}(GL_{n_1}(F) \times \ldots GL_{n_k}(F)) $ be a spherical function, and let $P = MN \subset GL_{n_1} \times \ldots GL_{n_k}$ be a parabolic subgroup. Let $\rho$ be an admissible representation of $M(F)$ of finite length, then we have $$\mathrm{Tr}(C_c^{GL_{n_1}(F) \times \ldots GL_{n_k}(F)}f, \mathrm{Ind}_P^{GL_{n_1}\times \ldots GL_{n_k}}(\rho)) = \mathrm{Tr}(C_c^{GL_{n_1}(F) \times \ldots GL_{n_k}(F)}f^{(P)}, \rho).$$ 
\end{lemma}

\begin{proof}
This follows from Proposition 3 of \cite{Kret11}. 
\end{proof}
In the following, let $(p_1,p_2)$ be a partition of $n$, let $s \in \mathbb Z_{>0} $ such that $s< n$. Let $\lambda = ((\lambda_1)^{p_1}, (\lambda_2)^{p_2}) \in \mathbb R_{\geq}^n$ such that $\lambda_1>\lambda_2$ and $\lambda_1 p_1 + \lambda_2p_2 = s$. We write $s_1 = \lambda_1p_1, s_2 = \lambda_2 p_2$, and we assume that $(s_1, p_1) = (s_2, p_2) = 1$. Then we have the following proposition:

\begin{proposition}
Let $\alpha \in \mathbb Z_{>0}$, and let $\pi =\mathrm{Ind}_P^{GL_n} (\otimes_{i = 1}^k \mathrm{Speh}(x_i, y)(\varepsilon_i| \cdot |^{e_i}))$ be a semi-stable rigid representation of $GL_n(F)$, then we have
\begin{itemize}
\item If $n$ is odd, or $n$ is even and $(p_1, p_2) \not = (\frac n2, \frac n2), (\frac n2 \pm 1, \frac n2 \mp 1)$, then $\mathrm{Tr}(C_{\lambda} f_{n \alpha s}, \pi) \not = 0 $ unless $\pi$ is a representation of $(p_1,p_2)$-type, or $\pi = \mathrm{1}_{GL_n(F)}(\varepsilon)$ where $\varepsilon$ is a unitary unramified character, or $\pi = \mathrm{St}_{GL_n(F)}(\varepsilon)$ where $\varepsilon$ is a unitary unramified character.
\item If $n$ is even and $(p_1, p_2) = (\frac n2, \frac n2)$ or $(p_1,p_2) = (\frac n2 \pm 1, \frac n2 \mp 1)$, then $\mathrm{Tr}(C_{\lambda} f_{n \alpha s}, \pi) \not = 0 $ unless $\pi$ is a representation of $(p_1,p_2)$-type, or $\pi = \mathrm{1}_{GL_n(F)}(\varepsilon)$ where $\varepsilon$ is a unitary unramified character, or $\pi = \mathrm{St}_{GL_n(F)}(\varepsilon)$ where $\varepsilon$ is a unitary unramified character, or $\pi = \mathrm{Speh}(\frac n2, 2)(\varepsilon)$ where $\varepsilon$ is a unitary unramified character, or $\pi = \mathrm{Speh}(2, \frac n2)(\varepsilon)$ where $\varepsilon$ is a unitary unramified character. 
\end{itemize}
\end{proposition}
 
\begin{proof}

By Proposition 3.3, we have $$\mathrm{Tr}(C_{\lambda} f_{n \alpha s}, \pi) =q^{\alpha(\frac{s(n-s)}2 - \sum\limits_{i=1}^2 \frac{s_i(p_i - s_i)}2)} \mathrm{Tr}(C^{GL_{p_1} \times GL_{p_2}(F)}_c f_{p_1 \alpha s_1} \otimes f_{p_2 \alpha s_2}, J_{N_{\lambda}}(\pi)).$$
Let $\rho$ denote $\otimes_{i = 1}^k \mathrm{Speh}(x_i, y)(\varepsilon_i| \cdot |^{e_i})$, and we assume $P$ corresponds to partition $\nu = (\nu_1, \ldots, \nu_k)$. By Theorem VI.5.1 of \cite{Renard10}, the semi-simplification of $J_{N_{\lambda}}(\pi)$   (denoted by $J_{N_{\lambda}}(\pi)_{s.s.}$) equals $$J_{N_{\lambda}}(\mathrm{Ind}_P^G(\rho))_{s.s.}= \bigoplus_{w \in S^{min}_{\nu, \lambda}  \subset S_n} \mathrm{Ind}_{M_{\lambda} \cap w^{-1} \cdot P_{\nu}}^{M_{\lambda}} (w \circ J_{w\cdot P_{\lambda} \cap M_{\nu}} (\rho))$$ where $ S^{min}_{\nu, \lambda}  \subset S_n$ is the set of minimal length representatives as defined after Proposition 4.2. By Proposition 4.3 and Proposition 4.4, we have the following: If $\nu \not = (p_1, p_2)$ and $\nu \not =  (n)$, then the length of $ w\cdot P_{\lambda} \cap M_{\nu}$ is greater than $2$ for all $w \in S^{min}_{\nu, \lambda}$ (therefore $  P_{\lambda} \cap w^{-1}\cdot P_{\nu} = w^{-1} \cdot (w \cdot P_{\lambda} \cap P_{\nu})$ is a proper subgroup of $P_{\lambda}$). By Lemma 5.1 and Proposition 5.1, we have that $$ \mathrm{Tr}(C^{GL_{p_1} \times GL_{p_2}(F)}_c f_{p_1 \alpha s_1} \otimes f_{p_2 \alpha s_2}, J_{N_{\lambda}}(\pi)) = $$
$$\mathrm{Tr}(C^{GL_{p_1} \times GL_{p_2}(F)}_c f_{p_1 \alpha s_1} \otimes f_{p_2 \alpha s_2},  \bigoplus_{w \in S^{min}_{\nu, \lambda}  \subset S_n} \mathrm{Ind}_{M_{\lambda} \cap w^{-1} \cdot P_{\nu}}^{M_{\lambda}} (w \circ J_{w\cdot P_{\lambda} \cap M_{\nu}} (\rho))) = 0$$
if $\nu \not = (p_1, p_2)$ and $\nu \not  = (n)$. 

If $\nu = (n)$, then $k = 1$ and $\pi =\mathrm{Ind}_P^{GL_n} (\mathrm{Speh}(x, y)(\varepsilon| \cdot |^{e}))$, by Theorem 5.4 of \cite{Tad95}, we have the following identity in the ring of Zelevinsky (as introduced in \cite{Zel80}), $$\mathrm {Speh}(x,y) = |\mathrm{det}|^{-\frac{x+ y}2} \sum\limits_{w \in S_y^{(x)}} (-1)^{\ell(w)} \prod_{i = 1}^y\rho_{i,w},$$ where $$S^{(x)}_y= \{w \in S_{y}: w(i) + x \geq i \text{ for all }1 \leq i \leq y\} ,$$ and $\rho_{i,w} = \mathrm{St}_{GL_{x +w(i)-i}(F)} (|\cdot|^{\frac{x + w(i) +i}2 -1} )$ is a representation of $GL_{x + w(i) - i} (F)$. 
If $x = 1$, then $\mathrm{Speh(1,n)}$ is the trivial representation, and if $y = 1$, then $\mathrm{Speh(n,1)}$ is the Steinberg representation. 

In the following, we assume that $x \geq 2$ and $y > 2$ (then $x + w(1) - 1>0$ and $x + w(2) - 2>0$). We also assume that there exists $w \in S_y^{(x)}$ such that $x+ w(i) - i = 0 $ for all $i$ with $3 \leq i \leq y$. Otherwise, if for all $w \in S_y^{(x)}$, we have $x+ w(i) - i > 0 $ for some $i$ with $3 \leq i \leq y$, then by Proposition 5.1 and Lemma 5.1, we have $\mathrm{Tr}(C^{GL_{p_1} \times GL_{p_2}(F)}_c f_{p_1 \alpha s_1} \otimes f_{p_2 \alpha s_2}, \mathrm{Speh}(x,y)) = 0$. By our assumption, we have $w(i) = i - x$ for all $i$ with $3 \leq i \leq y$, then we must have $x = 2$, and in this case, we have $w(1) = y, w(2) = y-1$ or $w(1) = y-1, w(2) = y$. Then we have $$\mathrm{Speh}(2,y) = |\mathrm{det}|^{-\frac{2+ y}2}(\mathrm{St}_{GL_{y+1}(F)}(|\cdot|^{\frac{y+1}2}) \times \mathrm{St}_{GL_{y-1}(F)}(|\cdot|^{\frac{y+1}2}) \pm \mathrm{St}_{GL_{y}(F)}(|\cdot|^{\frac{y}2}) \times \mathrm{St}_{GL_{y}(F)}(|\cdot|^{\frac{y+2}2}))$$
$$
=\mathrm{St}_{GL_{y+1}(F)}(|\cdot|^{-\frac{1}2}) \times \mathrm{St}_{GL_{y-1}(F)}(|\cdot|^{-\frac{1}2}) \pm \mathrm{St}_{GL_{y}(F)}(|\cdot|^{-1}) \times \mathrm{St}_{GL_{y}(F)},
$$
where we take the plus sign if $y >3$, and we take the minus sign if $y = 3$. By the previous discussion, we have $\mathrm{Tr}(C_{\lambda} f_{n\alpha s}, \mathrm{Speh}(2,y)) \not = 0$ unless $(p_1, p_2) = (\frac n2\pm1, \frac n2 \mp 1)$ or $(p_1, p_2) = (\frac n2, \frac n2 )$. 

If $x \geq 2$ and $y = 2$, then we have $$ \mathrm{Speh}(x,2) = |\mathrm{det}|^{-\frac{x+2}2}(\mathrm{St}_{G_{x}(F)}(|\cdot|^{\frac{x}2}) \times \mathrm{St}_{G_x(F)} (|\cdot|^{\frac{x+2}2})- \mathrm{St}_{G_{x+1}(F)} (|\cdot|^{\frac{x+1}2})\times \mathrm{St}_{G_{x-1}(F)}(|\cdot|^{\frac{x+1}2})$$
$$= \mathrm{St}_{G_{x}(F)}(|\cdot|^{-1}) \times \mathrm{St}_{G_x(F)} - \mathrm{St}_{G_{x+1}(F)} (|\cdot|^{-\frac{1}2})\times \mathrm{St}_{G_{x-1}(F)}(|\cdot|^{-\frac{1}2}).$$
And by the previous discussion, we have $\mathrm{Tr}(C_{\lambda} f_{n\alpha s}, \mathrm{Speh}(2,y)) \not = 0$ unless $(p_1, p_2) = (\frac n2\pm1, \frac n2 \mp 1)$ or $(p_1, p_2) = (\frac n2, \frac n2 )$.

If $\nu = (p_1, p_2)$, then $\pi =\mathrm{Ind}_P^{GL_n} (\otimes_{i = 1}^2 \mathrm{Speh}(x_i, y)(\varepsilon_i| \cdot |^{e_i}))$ where $p_1 = x_1 y, p_2 = x_2 y$. By repeating the proof of Proposition 8 in \cite{Kret11}, we may write $\pi$ as sum of representations (up to $\pm 1$) induced from proper parabolic subgroup of $GL_{p_1} \times GL_{p_2}$, unless $y  = 1$ or $x_1 = x_2 = 1$. Thus, the representation $\pi$ is a representation of $(p_1, p_2)$-type.  

\end{proof}

\begin{definition}

Let $\pi$ be an admissible representation of $GL_n(F)$, we say that $\pi$ is of type I if $\pi$ is a representation of $(p_1,p_2)$-type, or $\pi = \mathrm{1}_{GL_n(F)}(\varepsilon)$ where $\varepsilon$ is a unitary unramified character, or $\pi = \mathrm{St}_{GL_n(F)}(\varepsilon)$ where $\varepsilon$ is a unitary unramified character.

 We say that $\pi$ is of type II if $n$ is even, and $\pi$ is a representation of $(p_1,p_2)$-type, or $\pi = \mathrm{1}_{GL_n(F)}(\varepsilon)$ where $\varepsilon$ is a unitary unramified character, or $\pi = \mathrm{St}_{GL_n(F)}(\varepsilon)$ where $\varepsilon$ is a unitary unramified character, or $\pi = \mathrm{Speh}(\frac n2, 2)(\varepsilon)$ where $\varepsilon$ is a unitary unramified character, or $\pi = \mathrm{Speh}(2, \frac n2)(\varepsilon)$ where $\varepsilon$ is a unitary unramified character. 
\end{definition}

The representations listed in Proposition 5.2 make contributions at $p$ to the cohomology of certain strata of Kottwitz varieties (which we will define in $\S 6$). We now compute the truncated traces against representations listed in Proposition 5.2, before the computation, we first introduce the notion of Hecke matrix: 

\begin{definition}
Let $(n_1,\ldots,n_k) $ be a partition of $n$, and let $\pi$ be an unramified representation of $GL_{n_1}(F) \times \ldots GL_{n_k}(F)$. Let $\mathrm{Ind}_{P_{0,n}}^{GL_{n_1}\times \ldots \times GL_{n_k}}(|\mathrm{det}|^{c_1}\otimes\ldots \otimes |\mathrm{det}|^{c_n})$ (with $c_i \in \mathbb C$, and $\mathrm{Re}(c_{n_i+1}) \geq \mathrm{Re}(c_{n_i+2}) \geq \ldots \geq \mathrm{Re}(c_{n_{i+1}})$ for $1 \leq i \leq k-1$, we set $n_0 = 0$) be a principal series representation such that $\pi$ is the Langlands quotient of $\mathrm{Ind}_{P_{0,n}}^{GL_{n_1} \times \ldots GL_{n_k}}(|\mathrm{det}|^{c_1}\otimes\ldots \otimes |\mathrm{det}|^{c_n})$. Then we define the Hecke matrix associated with $\pi$ (denoted by $\epsilon_{\pi}$) to be $(q^{-c_1}, \ldots, q^{-c_n})$.
\end{definition}

Given $f \in \mathcal H^{\mathrm{unr}}(M_{(n_1, \ldots, n_k)}(F))$ where $(n_1, \dots, n_k)$ is a partition of $n$, we regard $\mathcal S(f)$ (where $\mathcal S$ is the Satake transform as introduced in $\S 1$) as an element of $\mathbb C[X_1^{\pm}, \ldots, X_n^{\pm}]$. For $(c_1, \ldots, c_n) \in \mathbb C^n$, let $\mathcal S (f)(c_1, \ldots, c_n)$ denote the evaluation of $\mathcal S(f) \in \mathbb C[X_1^{\pm}, \ldots, X_n^{\pm}]$ at $X_1 = c_1, \ldots, X_n = c_n$.
 
\begin{proposition} Let $\varepsilon: F^{\times} \rightarrow \mathbb C^{\times}$ be a unitary unramified character, let $c \in \mathbb C$ such that $\varepsilon(\pi_F) = |\pi_F|^{-c} = q^{c}$, then we have: 

(1) $\mathrm{Tr}(C_{\lambda}f_{n\alpha s},\mathrm{St}_{GL_n(F)}(\varepsilon)) = $
$$(-1)^{n-1} q^{\alpha(\frac{s(n-s)}2 - \sum\limits_{i=1}^2 \frac{s_i(p_i - s_i)}2)} q^{c \alpha s}\mathcal S(\hat{\chi}^{GL_{p_1}}_{M_{0,n}} f_{p_1 \alpha s_1 })(q^{\frac{1-n}2}, q^{\frac{3-n}2}, \ldots, q^{\frac{1 - n + 2(p_1 - 1)}2})$$

$$\times \mathcal S(\hat{\chi}^{GL_{p_2}}_{M_{0,n}} f_{p_2 \alpha s_2 })(q^{\frac{1-n +2p_1}2}, q^{\frac{1-n +2 (p_1 +2)}2}, \ldots, q^{\frac{n - 1 }2});$$

(2) $\mathrm{Tr}(C_{\lambda}f_{n\alpha s},\mathrm{1}_{GL_n(F)}(\varepsilon)) =$

$$q^{\alpha(\frac{s(n-s)}2 - \sum\limits_{i=1}^2 \frac{s_i(p_i - s_i)}2)} q^{c \alpha s}\mathcal S(\hat{\chi}^{GL_{p_1}}_{M_{0,p_1}} f_{p_1 \alpha s_1}) (q^{(1-p_1) + \frac{n-1}2}, q^{(3-p_1) + \frac{n-3}2}, \ldots, q^{(p_1-1) + \frac{n -1 - 2(p_1 - 1)}2})  $$
$$\times \mathcal S(\hat{\chi}^{GL_{p_2}}_{M_{0,p_2}} f_{p_2 \alpha s_2}) (q^{(1-p_2) + \frac{n-1 - 2p_1}2}, q^{(3-p_2) + \frac{n-1 - 2p_1 - 2}2}, \ldots, q^{(p_2-1) + \frac{ 1-n }2} ).$$ 

(3) If $n$ is even (with $n \geq 6$) and $(p_1,p_2) = ({\frac{n}2},\frac n2)$, then $\mathrm{Tr}(C_{\lambda}f_{n\alpha s},\mathrm{Speh}(2,y)(\varepsilon))$ 
$$= \pm q^{\alpha(\frac{s(n-s)}2 - \sum\limits_{i=1}^2 \frac{s_i(p_i - s_i)}2)} q^{c\alpha s}(\mathcal S(f_{y \alpha s_1})(q^{\frac{1-y}2}, q^{\frac{3-y}2}, \ldots, q^{\frac{y-1}2} )\mathcal S(f_{y \alpha s_2})(q^{\frac{-y}2}, q^{\frac{1-y}2}, \ldots, q^{\frac{y-2}2} ) $$

$$+ \mathcal S(f_{y \alpha s_2})(q^{\frac{1-y}2}, q^{\frac{3-y}2}, \ldots, q^{\frac{y-1}2} )\mathcal S(f_{y \alpha s_1})(q^{\frac{-y}2}, q^{\frac{1-y}2}, \ldots, q^{\frac{y-2}2} ));$$
and we take the plus sign if $y>3$, we take the minus sign if $y = 1$. 

If $n$ is even (with $n \geq 6$) and $(p_1,p_2) = ({\frac{n}2\pm 1},\frac n2 \mp1)$, then $\mathrm{Tr}(C_{\lambda}f_{n\alpha s},\mathrm{Speh}(2,y)(\varepsilon))$ 
$$= q^{\alpha(\frac{s(n-s)}2 - \sum\limits_{i=1}^2 \frac{s_i(p_i - s_i)}2)} q^{c\alpha s}(\mathcal S(f_{y+1, \alpha s_1})(q^{\frac{-y - 1}2}, q^{\frac{1-y}2}, \ldots, q^{\frac{y-1}2} )\mathcal S(f_{y-1, \alpha s_2})(q^{\frac{1-y}2}, q^{\frac{3-y}2}, \ldots, q^{\frac{y-3}2} ).$$

(4) If $n$ is even (with $n \geq 4$) and $(p_1,p_2) = ({\frac{n}2},\frac n2)$, then $\mathrm{Tr}(C_{\lambda}f_{n\alpha s},\mathrm{Speh}(x,2)(\varepsilon))$ 
$$=  q^{\alpha(\frac{s(n-s)}2 - \sum\limits_{i=1}^2 \frac{s_i(p_i - s_i)}2)}q^{c\alpha s}\mathcal S(\hat{\chi}^{GL_{p_1} }_{M_{0,p_1}} f_{p_1 \alpha s_1 } )(q^{\frac{-1-x}2}, q^{\frac{1-x}2}, \ldots, q^{\frac{x-3}2} )$$ 
$$\times \mathcal S(\hat{\chi}^{GL_{p_2}}_{M_{0,p_2}}  f_{p_2 \alpha s_2})(q^{\frac{1-x}2}, q^{\frac{3-x}2}, \ldots, q^{\frac{x-1}2} ).$$ 

If $n$ is even (with $n \geq 4$) and $(p_1,p_2) = ({\frac{n}2\pm 1},\frac n2 \mp1)$, then $\mathrm{Tr}(C_{\lambda}f_{n\alpha s},\mathrm{Speh}(x,2)(\varepsilon))$ 
$$= - q^{\alpha(\frac{s(n-s)}2 - \sum\limits_{i=1}^2 \frac{s_i(p_i - s_i)}2)}q^{c \alpha s}\mathcal S(\hat{\chi}^{GL_{p_1} }_{M_{0,p_1}} f_{p_1 \alpha s_1 } )(q^{\frac{-p_1}2}, q^{\frac{2-p_1}2}, \ldots, q^{\frac{p_1-2}2} ) $$
$$\times \mathcal S(\hat{\chi}^{GL_{p_2}}_{M_{0,p_2}}  f_{p_2 \alpha s_2})(q^{\frac{-p_2}2}, q^{\frac{2-p_2}2}, \ldots, q^{\frac{p_2-2}2} ).$$

Let $\varepsilon_{i}: F^{\times} \rightarrow \mathbb C^{\times}$ ($i = 1,2$) be a unitary unramified character, let $c_i \in \mathbb C$ such that $\varepsilon(\pi_F) = |\pi_F|^{-c_i} = q^{c_i}$ ($i = 1,2$), let $c \in (\frac 12, \frac 12) \subset \mathbb R$.

(5) If $n$ is odd, or $n$ is even and $p_1 \not = p_2$, then $\mathrm{Tr}(C_{\lambda}f_{n\alpha s},\mathrm{Ind}_{P}^{GL_n}(\mathrm{St}_{GL_{p_1}(F)}(\varepsilon_1)\otimes \mathrm{St}_{GL_{p_2}(F)}(\varepsilon_2)))$ 
$$=  (-1)^{n}q^{\alpha(\frac{s(n-s)}2 - \sum\limits_{i=1}^2 \frac{s_i(p_i - s_i)}2)}q^{c_1\alpha s_1}\mathcal S(\hat{\chi}^{GL_{p_1} }_{M_{0,p_1}} f_{p_1 \alpha s_1 } )(q^{\frac{1-p_1}2}, q^{\frac{3-p_1}2}, \ldots, q^{\frac{p_1-1}2} )$$ 
$$\times q^{c_2\alpha s_2}\mathcal S(\hat{\chi}^{GL_{p_2}}_{M_{0,p_2}}  f_{p_2 \alpha s_2})(q^{\frac{1-p_2}2}, q^{\frac{3-p_2}2}, \ldots, q^{\frac{p_2-1}2} );$$
if $n$ is even and $p_1 = p_2 = \frac n2$, then $\mathrm{Tr}(C_{\lambda}f_{n\alpha s},\mathrm{Ind}_{P}^{GL_n}(\mathrm{St}_{GL_{p_1}(F)}(\varepsilon|\cdot|^{r})\otimes \mathrm{St}_{GL_{p_2}(F)}(\varepsilon|\cdot|^{-r})))$ 
$$=  q^{\alpha(\frac{s(n-s)}2 - \sum\limits_{i=1}^2 \frac{s_i(p_i - s_i)}2)}q^{c\alpha s}(q^{r\alpha (s_2 - s_1)}\mathcal S(\hat{\chi}^{GL_{p_1} }_{M_{0,p_1}} f_{p_1 \alpha s_1 } )(q^{\frac{1-p_1}2}, q^{\frac{3-p_1}2}, \ldots, q^{\frac{p_1-1}2} )$$
$$\times \mathcal S(\hat{\chi}^{GL_{p_2}}_{M_{0,p_2}}  f_{p_2 \alpha s_2})(q^{\frac{1-p_2}2}, q^{\frac{3-p_2}2}, \ldots, q^{\frac{p_2-1}2} )$$
$$+q^{r\alpha (s_1-s_2)}\mathcal S(\hat{\chi}^{GL_{p_1} }_{M_{0,p_1}} f_{p_1 \alpha s_1 } )(q^{\frac{1-p_1}2}, q^{\frac{3-p_1}2}, \ldots, q^{\frac{p_1-1}2} ) 
 q^{c_1\alpha s_2}\mathcal S(\hat{\chi}^{GL_{p_2}}_{M_{0,p_2}}  f_{p_2 \alpha s_2})(q^{\frac{1-p_2}2}, q^{\frac{3-p_2}2}, \ldots, q^{\frac{p_2-1}2} ));$$ and $\mathrm{Tr}(C_{\lambda}f_{n\alpha s},\mathrm{Ind}_{P}^{GL_n}(\mathrm{St}_{GL_{p_1}(F)}(\varepsilon_1)\otimes \mathrm{St}_{GL_{p_2}(F)}(\varepsilon_2)))$ 
$$=  q^{\alpha(\frac{s(n-s)}2 - \sum\limits_{i=1}^2 \frac{s_i(p_i - s_i)}2)}(q^{c_1\alpha s_1}\mathcal S(\hat{\chi}^{GL_{p_1} }_{M_{0,p_1}} f_{p_1 \alpha s_1 } )(q^{\frac{1-p_1}2}, q^{\frac{3-p_1}2}, \ldots, q^{\frac{p_1-1}2} )$$
$$\times q^{c_2\alpha s_2}\mathcal S(\hat{\chi}^{GL_{p_2}}_{M_{0,p_2}}  f_{p_2 \alpha s_2})(q^{\frac{1-p_2}2}, q^{\frac{3-p_2}2}, \ldots, q^{\frac{p_2-1}2} )$$
$$+q^{c_2\alpha s_1}\mathcal S(\hat{\chi}^{GL_{p_1} }_{M_{0,p_1}} f_{p_1 \alpha s_1 } )(q^{\frac{1-p_1}2}, q^{\frac{3-p_1}2}, \ldots, q^{\frac{p_1-1}2} ) 
 q^{c_1\alpha s_2}\mathcal S(\hat{\chi}^{GL_{p_2}}_{M_{0,p_2}}  f_{p_2 \alpha s_2})(q^{\frac{1-p_2}2}, q^{\frac{3-p_2}2}, \ldots, q^{\frac{p_2-1}2} )).$$

(6) If $n$ is odd, or $n$ is even and $p_1 \not = p_2$, then $\mathrm{Tr}(C_{\lambda}f_{n\alpha s},\mathrm{Ind}_{P}^{GL_n}(\mathrm{1}_{GL_{p_1}(F)}(\varepsilon_1)\otimes \mathrm{1}_{GL_{p_2}(F)}(\varepsilon_2)))$ 
$$=  (-1)^{n}q^{\alpha(\frac{s(n-s)}2 - \sum\limits_{i=1}^2 \frac{s_i(p_i - s_i)}2)}q^{c_1\alpha s_1}\mathcal S(\hat{\chi}^{GL_{p_1} }_{M_{0,p_1}} f_{p_1 \alpha s_1 } )(q^{\frac{1-p_1}2}, q^{\frac{3-p_1}2}, \ldots, q^{\frac{p_1-1}2} )$$ 
$$\times q^{c_2\alpha s_2}\mathcal S(\hat{\chi}^{GL_{p_2}}_{M_{0,p_2}}  f_{p_2 \alpha s_2})(q^{\frac{1-p_2}2}, q^{\frac{3-p_2}2}, \ldots, q^{\frac{p_2-1}2} );$$

if $n$ is even and $p_1 = p_2 = \frac n2$, then $\mathrm{Tr}(C_{\lambda}f_{n\alpha s},\mathrm{Ind}_{P}^{GL_n}(\mathrm{1}_{GL_{p_1}(F)}(\varepsilon|\cdot|^{r})\otimes \mathrm{1}_{GL_{p_2}(F)}(\varepsilon|\cdot|^{-r})))$ 
$$=  q^{\alpha(\frac{s(n-s)}2 - \sum\limits_{i=1}^2 \frac{s_i(p_i - s_i)}2)}q^{c\alpha s}(q^{r\alpha (s_2 - s_1)}\mathcal S(\hat{\chi}^{GL_{p_1} }_{M_{0,p_1}} f_{p_1 \alpha s_1 } )(q^{\frac{1-p_1}2}, q^{\frac{3-p_1}2}, \ldots, q^{\frac{p_1-1}2} )$$
$$\times \mathcal S(\hat{\chi}^{GL_{p_2}}_{M_{0,p_2}}  f_{p_2 \alpha s_2})(q^{\frac{1-p_2}2}, q^{\frac{3-p_2}2}, \ldots, q^{\frac{p_2-1}2} )$$
$$+q^{r \alpha (s_1-s_2)}\mathcal S(\hat{\chi}^{GL_{p_1} }_{M_{0,p_1}} f_{p_1 \alpha s_1 } )(q^{\frac{1-p_1}2}, q^{\frac{3-p_1}2}, \ldots, q^{\frac{p_1-1}2} ) 
 q^{c_1\alpha s_2}\mathcal S(\hat{\chi}^{GL_{p_2}}_{M_{0,p_2}}  f_{p_2 \alpha s_2})(q^{\frac{1-p_2}2}, q^{\frac{3-p_2}2}, \ldots, q^{\frac{p_2-1}2} ));$$ and $\mathrm{Tr}(C_{\lambda}f_{n\alpha s},\mathrm{Ind}_{P}^{GL_n}(\mathrm{1}_{GL_{p_1}(F)}(\varepsilon_1)\otimes \mathrm{1}_{GL_{p_2}(F)}(\varepsilon_2)))$ 
$$=  q^{\alpha(\frac{s(n-s)}2 - \sum\limits_{i=1}^2 \frac{s_i(p_i - s_i)}2)}(q^{c_1\alpha s_1}\mathcal S(\hat{\chi}^{GL_{p_1} }_{M_{0,p_1}} f_{p_1 \alpha s_1 } )(q^{\frac{1-p_1}2}, q^{\frac{3-p_1}2}, \ldots, q^{\frac{p_1-1}2} )$$
$$\times q^{c_2\alpha s_2}\mathcal S(\hat{\chi}^{GL_{p_2}}_{M_{0,p_2}}  f_{p_2 \alpha s_2})(q^{\frac{1-p_2}2}, q^{\frac{3-p_2}2}, \ldots, q^{\frac{p_2-1}2} )$$
$$+q^{c_2\alpha s_1}\mathcal S(\hat{\chi}^{GL_{p_1} }_{M_{0,p_1}} f_{p_1 \alpha s_1 } )(q^{\frac{1-p_1}2}, q^{\frac{3-p_1}2}, \ldots, q^{\frac{p_1-1}2} ) 
 q^{c_1\alpha s_2}\mathcal S(\hat{\chi}^{GL_{p_2}}_{M_{0,p_2}}  f_{p_2 \alpha s_2})(q^{\frac{1-p_2}2}, q^{\frac{3-p_2}2}, \ldots, q^{\frac{p_2-1}2} )).$$

\end{proposition}

\begin{proof}
Let $\pi$ be a semi-stable rigid representation of $GL_n(F)$, by applying Proposition 3.4, we have $$\mathrm{Tr}(C_{\lambda} f_{n\alpha s}, \pi) =q^{\alpha(\frac{s(n-s)}2 - \sum\limits_{i=1}^2 \frac{s_i(p_i - s_i)}2)}\mathrm{Tr}(C^{M_{(p_1,p_2)}(F)}_{c}  (f_{p_1 \alpha s_1} \otimes f_{p_2 \alpha s_2}), \pi_{N_{(p_1, p_2)}}(\delta_{GL_n, P_{(p_1,p_2)}}^{-\frac 12}) ).$$

(1) If $\pi = \mathrm{St}_{GL_n(F)}(\varepsilon)$, by Proposition 2 of \cite{Rod86}, we have $$\mathrm{Tr}(C_{\lambda} f_{n\alpha s}, \pi) =q^{\alpha(\frac{s(n-s)}2 - \sum\limits_{i=1}^2 \frac{s_i(p_i - s_i)}2)}\mathrm{Tr}(C^{M_{(p_1,p_2)}(F)}_{c}  (f_{p_1 \alpha s_1} \otimes f_{p_2 \alpha s_2}), \mathrm{St}_{M_{(p_1, p_2)}(F)}(\delta_{GL_n, P_{(p_1,p_2)}}^{\frac 12} \cdot \varepsilon) ).$$ By repeating the proof of Proposition 6 in \cite{Kret11}, we have $$\mathrm{Tr}(C^{M_{(p_1,p_2)}(F)}_{c}  (f_{p_1 \alpha s_1} \otimes f_{p_2 \alpha s_2}), \mathrm{St}_{M_{(p_1, p_2)}(F)}(\delta_{GL_n, P_{(p_1,p_2)}}^{\frac 12} \cdot \varepsilon) )$$ 
$$= (-1)^{p_1 + p_2 - 2}\mathrm{Tr}(\hat{\chi}^{GL_{p_1} \times GL_{p_2}}_{M_{0,n}} f_{p_1 \alpha s_1} \otimes f_{p_2 \alpha s_2}, \delta_{GL_n, P_{(p_1,p_2)}}^{\frac 12}|_{P_{0,n}(F)} \cdot \delta_{ M_{(p_1,p_2)}, P_{0,n}}^{\frac 12 } ( \varepsilon))$$

$$= (-1)^n\mathrm{Tr}(\hat{\chi}^{GL_{p_1} \times GL_{p_2}}_{M_{0,n}} f_{p_1 \alpha s_1} \otimes f_{p_2 \alpha s_2}, \delta_{ GL_n, P_{0,n}}^{\frac 12}( \varepsilon)) .$$ By Lemma 7.7 of \cite{Kret15}, we have $$\mathrm{Tr}(\hat{\chi}^{GL_{p_1} \times GL_{p_2}}_{M_{0,n}} f_{p_1 \alpha s_1} \otimes f_{p_2 \alpha s_2}, \delta_{ GL_n, P_{0,n}}^{\frac 12}( \varepsilon)) = \mathcal S(\hat{\chi}^{GL_{p_1} \times GL_{p_2}}_{M_{0,n}} f_{p_1 \alpha s_1} \otimes f_{p_2 \alpha s_2})(\epsilon_{\delta_{ GL_n, P_{0,n}( \varepsilon)}^{\frac 12}})$$ with $\epsilon_{\delta_{ GL_n, P_{0,n}}^{\frac 12}} = (q^{c\frac{1-n}2}, q^{c\frac{3-n}2}, \ldots, q^{c\frac{n-1}2})$.  Also note that $\mathcal S(\hat{\chi}^{GL_{p_1} \times GL_{p_2}}_{M_{0,n}} f_{p_1 \alpha s_1} \otimes f_{p_2 \alpha s_2})$ is homogenous of degree $\alpha s$, and $$\mathcal S(\hat{\chi}^{GL_{p_1} \times GL_{p_2}}_{M_{0,n}} f_{p_1 \alpha s_1} \otimes f_{p_2 \alpha s_2}) = \mathcal S(\hat{\chi}^{GL_{p_1}}_{M_{0,p_1}} f_{p_1 \alpha s_1}) \otimes \mathcal S(\hat{\chi}^{GL_{p_2}}_{M_{0,p_2}} f_{p_2 \alpha s_2}), $$ then we proved (1). 

(2) If $\pi = \mathrm{1}_{GL_n(F)}(\varepsilon)$, we have $$\mathrm{Tr}(C_{\lambda} f_{n\alpha s}, \pi) =q^{\alpha(\frac{s(n-s)}2 - \sum\limits_{i=1}^2 \frac{s_i(p_i - s_i)}2)}\mathrm{Tr}(C^{M_{(p_1,p_2)}(F)}_{c}  (f_{p_1 \alpha s_1} \otimes f_{p_2 \alpha s_2}), \mathrm{1}_{M_{(p_1, p_2)}(F)}(\delta_{GL_n, P_{(p_1,p_2)}}^{-\frac 12} \cdot \varepsilon) ).$$ By repeating the proof of Proposition 7 in \cite{Kret11}, we have $$\mathrm{Tr}(C^{M_{(p_1,p_2)}(F)}_{c}  (f_{p_1 \alpha s_1} \otimes f_{p_2 \alpha s_2}), \mathrm{1}_{M_{(p_1, p_2)}(F)}(\delta_{GL_n, P_{(p_1,p_2)}}^{-\frac 12} \cdot \varepsilon) )$$ 
$$= (-1)^{p_1 + p_2 - 2}\mathrm{Tr}(C^{M_{(p_1,p_2)}(F)}_{c}  (f_{p_1 \alpha s_1} \otimes f_{p_2 \alpha s_2}), \mathrm{St}_{M_{(p_1, p_2)}(F)}(\delta_{GL_n, P_{(p_1,p_2)}}^{-\frac 12} \cdot \varepsilon) )$$

$$= \mathrm{Tr}(\hat{\chi}^{GL_{p_1} \times GL_{p_2}}_{M_{0,n}} f_{p_1 \alpha s_1} \otimes f_{p_2 \alpha s_2}, \delta_{GL_n, P_{(p_1,p_2)}}^{-\frac 12}|_{P_{0,n}(F)} \cdot \delta_{ M_{(p_1,p_2)}, P_{0,n}}^{\frac 12 } ( \varepsilon)) .$$

$$=\mathcal S(\hat{\chi}^{GL_{p_1} \times GL_{p_2}}_{M_{0,n}} f_{p_1 \alpha s_1} \otimes f_{p_2 \alpha s_2})(\epsilon_{\delta_{GL_n, P_{(p_1,p_2)}}^{-\frac 12}|_{P_{0,n}(F)} \cdot \delta_{ M_{(p_1,p_2)}, P_{0,n}}^{\frac 12 } ( \varepsilon)})$$
 with $$\epsilon_{\delta_{GL_n, P_{(p_1,p_2)}}^{-\frac 12}|_{P_{0,n}(F)} \cdot \delta_{ M_{(p_1,p_2)}, P_{0,n}}^{\frac 12 } ( \varepsilon)} =\epsilon_{\delta_{GL_n, P_{0,n}}^{-\frac 12} \cdot \delta_{ M_{(p_1,p_2)}, P_{0,n}} ( \varepsilon)}$$ 
 $$(q^{c(1-p_1) + c\frac{n-1}2}, q^{c(3-p_1) + c\frac{n-3}2}, \ldots, q^{c(p_1-1) + c\frac{n -1 - 2(p_1 - 1)}2},$$
 
 $$q^{c(1-p_2) + c\frac{n-1 - 2p_1}2}, q^{c(3-p_2) + c\frac{n-1 - 2p_1 - 2}2}, \ldots, q^{c(p_2-1) + c\frac{ 1-n }2} ).$$ Also note that $\mathcal S(\hat{\chi}^{GL_{p_1} \times GL_{p_2}}_{M_{0,n}} f_{p_1 \alpha s_1} \otimes f_{p_2 \alpha s_2})$ is homogenous of degree $\alpha s$, and $$\mathcal S(\hat{\chi}^{GL_{p_1} \times GL_{p_2}}_{M_{0,n}} f_{p_1 \alpha s_1} \otimes f_{p_2 \alpha s_2}) = \mathcal S(\hat{\chi}^{GL_{p_1}}_{M_{0,p_1}} f_{p_1 \alpha s_1}) \otimes \mathcal S(\hat{\chi}^{GL_{p_2}}_{M_{0,p_2}} f_{p_2 \alpha s_2}),$$ then we proved (2). 
 
(3) If $\pi = \mathrm{Speh}(2,y)(\varepsilon)$, and if $n$ is even and $(p_1,p_2) = ({\frac{n}2},\frac n2) =(y,y)$. From the proof of Proposition 5.2, we have $$\mathrm{Tr}(C^{M_{(p_1,p_2)}(F)}_{c}  (f_{p_1 \alpha s_1} \otimes f_{p_2 \alpha s_2}), \pi_{N_{(p_1,p_2)}} (\delta_{GL_n, P_{y,y}}^{-\frac 12})) = $$
$$\pm \mathrm{Tr}(C^{M_{(y,y)}(F)}_{c}  (f_{y \alpha s_1} \otimes f_{y \alpha s_2}), J_{N_{(y,y)}}(\mathrm{Ind}_{P_{(y,y)}}^{GL_n}(\mathrm{St}_{GL_y(F)}\otimes \mathrm{St}_{GL_y(F)})(\varepsilon) ));$$ such that we take the plus sign if $y>3$, we take the minus sign if $ y = 3$. By Theorem VI.5.1 of \cite{Renard10} (see also the proof of Proposition 5.2), Proposition 4.4 and our computation of $J_{N_{(y,y)}}(\mathrm{Ind}_{P_{(y,y)}}^{GL_n}(\mathrm{St}_{GL_y(F)}\otimes \mathrm{St}_{GL_y(F)})(\varepsilon) )$ in the proof of Proposition 5.2, we have $$\mathrm{Tr}(C^{M_{(y,y)}(F)}_{c}  (f_{y \alpha s_1} \otimes f_{y \alpha s_2}), J_{N_{(y,y)}}(\mathrm{Ind}_{P_{(y,y)}}^{GL_n}(\mathrm{St}_{GL_y(F)}(|\cdot|^{-1})\otimes \mathrm{St}_{GL_y(F)})(\varepsilon))) = $$

$$\mathrm{Tr}(C^{M_{(y,y)}(F)}_{c}  (f_{y \alpha s_1} \otimes f_{y \alpha s_2}), \mathrm{St}_{GL_y(F)}(\varepsilon|\cdot|^{-1})\otimes \mathrm{St}_{GL_y(F)}(\varepsilon)) $$

$$+ \mathrm{Tr}(C^{M_{(y,y)}(F)}_{c}  (f_{y \alpha s_1} \otimes f_{y \alpha s_2}), w \cdot \mathrm{St}_{GL_y(F)}(\varepsilon|\cdot|^{-1})\otimes \mathrm{St}_{GL_y(F)}(\varepsilon))$$
(where $w \in S_n$ is the permutation sending $1, 2, \ldots, y$ to $y+1, \ldots, 2y $ respectively)

$$ = \mathrm{Tr}(C^{M_{(y,y)}(F)}_{c}  (f_{y \alpha s_1} \otimes f_{y \alpha s_2}), \mathrm{St}_{GL_y(F)}(\varepsilon|\cdot|^{-1})\otimes \mathrm{St}_{GL_y(F)}(\varepsilon)) $$

$$+ \mathrm{Tr}(C^{M_{(y,y)}(F)}_{c}  (f_{y \alpha s_1} \otimes f_{y \alpha s_2}), \mathrm{St}_{GL_y(F)}(\varepsilon)\otimes \mathrm{St}_{GL_y(F)}(\varepsilon|\cdot|^{-1}))$$
Repeat the proof of Proposition 6 in \cite{Kret11}, the above equality equals

$$q^{c\alpha s}(\mathcal S(f_{y \alpha s_1})(q^{\frac{1-y}2}, q^{\frac{3-y}2}, \ldots, q^{\frac{y-1}2} )\mathcal S(f_{y \alpha s_2})(q^{\frac{-y}2}, q^{\frac{1-y}2}, \ldots, q^{\frac{y-2}2} ) $$

$$+ \mathcal S(f_{y \alpha s_2})(q^{\frac{1-y}2}, q^{\frac{3-y}2}, \ldots, q^{\frac{y-1}2} )\mathcal S(f_{y \alpha s_1})(q^{\frac{-y}2}, q^{\frac{1-y}2}, \ldots, q^{\frac{y-2}2} )).$$
The case when $(p_1,p_2) = ({\frac{n}2\pm 1},\frac n2 \mp1)$ is similar to the case when $(p_1,p_2) = ({\frac{n}2},\frac n2 )$.

(4) The proof of (4) is identical to the proof of $(3)$. 

(5) We first consider the case when $n $ is even and $p_1 = p_2$, and we assume $$\pi = \mathrm{Ind}_{P}^{GL_n}(\mathrm{St}_{GL_{p_1}(F)}(\varepsilon|\cdot|^{r})\otimes \mathrm{St}_{GL_{p_2}(F)}(\varepsilon|\cdot|^{-r}))).$$ We only need to compute $$\mathrm{Tr}(C^{M_{(\frac n2,\frac n2)}(F)}_{c}  (f_{\frac n2 \alpha s_1} \otimes f_{\frac n2 \alpha s_2}), J_{N_{(\frac n2,\frac n2)}}(\mathrm{Ind}_{P_{(\frac n2,\frac n2)}}^{GL_n}(\mathrm{St}_{GL_{\frac n2}(F)}(\varepsilon|\cdot|^{r})\otimes \mathrm{St}_{GL_{\frac n2}(F)}(\varepsilon|\cdot|^{-r})))).$$ By Theorem VI 5.4 in \cite{Renard10} and Proposition 4.4, we have 
$$J_{N_{(\frac n2,\frac n2)}}(\mathrm{Ind}_{P_{(\frac n2,\frac n2)}}^{GL_n}(\mathrm{St}_{GL_{\frac n2}(F)}(\varepsilon|\cdot|^{r})\otimes \mathrm{St}_{GL_{\frac n2}(F)}(\varepsilon|\cdot|^{-r}))) =$$
$$\mathrm{St}_{GL_{\frac n2}(F)}(\varepsilon|\cdot|^{r})\otimes \mathrm{St}_{GL_{\frac n2}(F)}(\varepsilon|\cdot|^{-r}) + w \cdot \mathrm{St}_{GL_{\frac n2}(F)}(\varepsilon|\cdot|^{r})\otimes \mathrm{St}_{GL_{\frac n2}(F)}(\varepsilon|\cdot|^{-r}),$$ where $w $ is the permutation sending $1, 2, \ldots \frac n2$ to $\frac n2+1, \frac n2+2, \ldots n$ respectively. Then we have $$\mathrm{Tr}(C^{M_{(\frac n2,\frac n2)}(F)}_{c}  (f_{\frac n2 \alpha s_1} \otimes f_{\frac n2 \alpha s_2}), J_{N_{(\frac n2,\frac n2)}}(\mathrm{Ind}_{P_{(\frac n2,\frac n2)}}^{GL_n}(\mathrm{St}_{GL_{\frac n2}(F)}(\varepsilon|\cdot|^{r})\otimes \mathrm{St}_{GL_{\frac n2}(F)}(\varepsilon|\cdot|^{-r}))))$$

$$ = \mathrm{Tr}(C^{M_{(\frac n2,\frac n2)}(F)}_{c}  (f_{\frac n2 \alpha s_1} \otimes f_{\frac n2 \alpha s_2}), \mathrm{St}_{GL_{\frac n2}(F)}(\varepsilon|\cdot|^{r})\otimes \mathrm{St}_{GL_{\frac n2}(F)}(\varepsilon|\cdot|^{-r}))$$ 
$$+ \mathrm{Tr}(C^{M_{(\frac n2,\frac n2)}(F)}_{c}  (f_{\frac n2 \alpha s_1} \otimes f_{\frac n2 \alpha s_2}),w \cdot \mathrm{St}_{GL_{\frac n2}(F)}(\varepsilon|\cdot|^{r})\otimes \mathrm{St}_{GL_{\frac n2}(F)}(\varepsilon|\cdot|^{-r})).$$ By repeating the proof of Proposition 6 in \cite{Kret11}, we proved our results. The proofs of the rest of the cases in (5) are similar. 

(6) We first consider the case when $n $ is even and $p_1 = p_2$, and we assume $$\pi = \mathrm{Ind}_{P}^{GL_n}(\mathrm{1}_{GL_{p_1}(F)}(\varepsilon|\cdot|^{r})\otimes \mathrm{1}_{GL_{p_2}(F)}(\varepsilon|\cdot|^{-r}))).$$ We only need to compute $$\mathrm{Tr}(C^{M_{(\frac n2,\frac n2)}(F)}_{c}  (f_{\frac n2 \alpha s_1} \otimes f_{\frac n2 \alpha s_2}), J_{N_{(\frac n2,\frac n2)}}(\mathrm{Ind}_{P_{(\frac n2,\frac n2)}}^{GL_n}(\mathrm{1}_{GL_{\frac n2}(F)}(\varepsilon|\cdot|^{r})\otimes \mathrm{1}_{GL_{\frac n2}(F)}(\varepsilon|\cdot|^{-r})))).$$ By Theorem VI 5.4 in \cite{Renard10} and Proposition 4.4, we have 
$$J_{N_{(\frac n2,\frac n2)}}(\mathrm{Ind}_{P_{(\frac n2,\frac n2)}}^{GL_n}(\mathrm{1}_{GL_{\frac n2}(F)}(\varepsilon|\cdot|^{r})\otimes \mathrm{1}_{GL_{\frac n2}(F)}(\varepsilon|\cdot|^{-r}))) =$$
$$\mathrm{1}_{GL_{\frac n2}(F)}(\varepsilon|\cdot|^{r})\otimes \mathrm{1}_{GL_{\frac n2}(F)}(\varepsilon|\cdot|^{-r}) + w \cdot \mathrm{1}_{GL_{\frac n2}(F)}(\varepsilon|\cdot|^{r})\otimes \mathrm{1}_{GL_{\frac n2}(F)}(\varepsilon|\cdot|^{-r})$$ where $w $ is the permutation sending $1, 2, \ldots \frac n2$ to $\frac n2+1, \frac n2+2, \ldots n$ respectively. Then we have $$\mathrm{Tr}(C^{M_{(\frac n2,\frac n2)}(F)}_{c}  (f_{\frac n2 \alpha s_1} \otimes f_{\frac n2 \alpha s_2}), J_{N_{(\frac n2,\frac n2)}}(\mathrm{Ind}_{P_{(\frac n2,\frac n2)}}^{GL_n}(\mathrm{1}_{GL_{\frac n2}(F)}(\varepsilon|\cdot|^{r})\otimes \mathrm{1}_{GL_{\frac n2}(F)}(\varepsilon|\cdot|^{-r}))))$$

$$ = \mathrm{Tr}(C^{M_{(\frac n2,\frac n2)}(F)}_{c}  (f_{\frac n2 \alpha s_1} \otimes f_{\frac n2 \alpha s_2}), \mathrm{1}_{GL_{\frac n2}(F)}(\varepsilon|\cdot|^{r})\otimes \mathrm{1}_{GL_{\frac n2}(F)}(\varepsilon|\cdot|^{-r}))$$ 
$$+ \mathrm{Tr}(C^{M_{(\frac n2,\frac n2)}(F)}_{c}  (f_{\frac n2 \alpha s_1} \otimes f_{\frac n2 \alpha s_2}),w \cdot \mathrm{1}_{GL_{\frac n2}(F)}(\varepsilon|\cdot|^{r})\otimes \mathrm{1}_{GL_{\frac n2}(F)}(\varepsilon|\cdot|^{-r})).$$ By repeating the proof of Proposition 7 and Proposition 6 in \cite{Kret11}, we proved our results. The proofs of the rest of the cases in (6) are similar. 

\end{proof}

\section{Cohomology of certain Newton strata of Kottwitz varieties}

 Having established the necessary local computation in $\S 5$, we now turn to the global setting. Firstly, we introduce Kottwitz varieties and certain Newton strata of Kottwitz varieties:

	\begin{enumerate}
	      
    \item Let $D$ be a division algebra over $\mathbb Q $, and let $*$ be an anti-involution. Let $\overline{\mathbb Q} $ be the algebraic closure of $\mathbb Q$ inside $\mathbb C$. Let $F$ be the center of $D$, and we embed $F$ into $\overline{\mathbb Q}$. We assume that $F$ is a quadratic imaginary field  extension of $\mathbb Q$ and we assume that $*$ induces the complex conjugation on $F$.
	
 \item Let $G$ be the $\mathbb Q$-group defined as follows: Let $R$ be a commutative $\mathbb Q$-algebra, then $G(R)$ is defined to be
	$\{x \in D \otimes_{\mathbb Q} R: xx^* \in R^{\times}\}$. 
	
\item Let $h_0$ be an $\mathbb R$-algebra morphism $h_0: \mathbb C \rightarrow D_{\mathbb R}$ such that $h_0(z)^* = h_0(\overline z)$ for all $z \in \mathbb C$. And we assume that the involution of $D_{\mathbb R	}$ defined by $x \mapsto h_0(i)^{-1}xh_0(i)$ (for $x\in D_{\mathbb R}$) is positive. 

\item The Shimura datum: We use $h'$ to denote $h_0|_{\mathbb C^{\times}}$, then $h'$ is an algebraic group morphism from the Deligne torus $\mathbb S$ to $G_{\mathbb R}$. Let $h$ denote $h'^{-1}$ and let $X$ denote the $G(\mathbb R$)-conjugacy class of $h$, then the pair $(G,X)$ is a PEL-type Shimura datum (see \cite{Kot92}).

 \item Let $n$ be the positive integer such that $n^2$ equals the dimension of $D$ over $F$.

 \item The group structure of $G$ at $\mathbb R$: The group $G_{\mathbb R}$ is isomorphic to $ GU_{\mathbb C/\mathbb R}(s, n-s)$ for some $0 \leq  s \leq n $ (with $GU_{\mathbb C/\mathbb R}(s, n-s)$ as defined in $\S2.1$ of \cite{Morel08}).

\item The group structures of $G$ at $p$: Let $p$ be a prime number such that $G_{\mathbb Q_p}$ is unramified, and let $K \subset G(\mathbb A_f)$ be a compact open subgroup. such that $K$ decomposes into a product $K_pK^p$ where $K_p \subset G(\mathbb Q_p)$ is hyperspecial and $K^p \subset G(\mathbb A_f^p)$ is small enough (so that $Sh_K(G,X)$ is smooth). We assume that $p$ splits in extension $F/\mathbb Q$, therefore, the group $G_{\mathbb Q_p}$ is isomorphic to $\mathbb G_{m,\mathbb Q_p} \times GL_{n,\mathbb Q_p} $. 

\item Kottwitz varieties: Let $Sh_K(G,X)$ be the Shimura variety defined over $E$ (which is called the Kottwitz variety). Let $M_K$ be the moduli scheme over $\mathcal O_E \otimes_{\mathbb Z} \mathbb Z_{(p)}$ which represents the corresponding moduli problem of Abelian varieties of PEL-type (as defined in $\S 5$ of \cite{kot92p}, note that $M_{K,E}$ is $|\mathrm{ker}^1(\mathbb Q, G)|$ copies of $Sh_K(G,X)$).

 \item Reflex field: We use $E$ to denote the reflex field of the Shimura datum $(G,X)$. If $n-s \not = s $, then the reflex field is $F$, and if $n-s = s $, then the reflex field is $\mathbb Q$ {see \cite{Kot92}}.	Let $\mathfrak p$ be a prime ideal of $E$ over $p$, and let $E_{\mathfrak p, \alpha}$ denote an unramified extension of $E_{\mathfrak p}$ of degree $\alpha$.

\item The set $B(G_{\mathbb Q_p},\mu)$: Let $B(G_{\mathbb Q_p})$ denote the set of isocrystals with $G_{\mathbb Q_p}$-structure, which can be identified with the set of $\sigma$-conjugacy class in $G(L)$ (see \cite{Kot85}, where $L = \widehat{ \mathbb Q_p^{\mathrm{unr}}}$ is the completion of the maximal unramified extension of $\mathbb Q_p$). Note that by our assumption, the group $G_{\mathbb Q_p} $ is isomorphic to $\mathbb G_{m,\mathbb Q_p} \times GL_{n,\mathbb Q_p} $. Let $\nu: B(G_{\mathbb Q}) \rightarrow \mathbb R^{n+1}$ denote the Newton map (as defined in \cite{Ra02}), and we define $\mu = (s, \mu_1, \ldots, \mu_n) = (s, 1^s, 0^{n-s})$. We define $B(G_{\mathbb Q_p},\mu)$ to be $\{b \in B(G_{\mathbb Q_p}): \nu(b) = (s, \lambda_1, \ldots, \lambda_n) \in \mathbb R^{n+1} \text{ such that }(\lambda_1, \ldots, \lambda_n) \in \mathbb R^n_{\geq} \text{ satisfies that } \sum_{j = 1}^i\lambda_j \geq \sum_{j = 1}^i\mu_j \text{ for } 1 \leq i \leq n \text{ and }\sum_{i = 1}^n \lambda_i = s\}$. Note that the set $B(G_{\mathbb Q_p}, \mu)$ is finite (see $\S6$ of \cite{Kot97} or $\S1$ of \cite{Ra02}). 

\item Newton stratification: We write $q$ for the number of elements in the residue field of $\mathcal O_{E_{\mathfrak p}}$. 
 For each geometric point $x \in M_{K, \overline{\mathbb F}_q}$, let $b(\mathcal A_x) \in B(G_{\mathbb Q_p})$ denote the isocrystal with $G$-structure associated to the abelian variety $\mathcal A_x$ corresponding to $x$ (see pp.170-171 of \cite{RaRi96}), and we have that $b(\mathcal A_x) \in B(G_{\mathbb Q_p},\mu)$. For $b \in B(G_{\mathbb Q_p},\mu )$, we define $S_b$ to be $\{x \in M_{K, \overline{\mathbb F}_q}(\overline {\mathbb F}_q): b(\mathcal A_x) = b\}$. Then $S_b$ is a subvariety of $M_{K, \overline{\mathbb F}_q}$ (see $\S 1$ of \cite{Ra02}) and the collection $\{S_b\}_{b\in B(G_{\mathbb Q_p},\mu)}$ of subvarieties of $M_{K, \overline{\mathbb F}_q}$ is called the Newton stratification of $M_{K, \overline{\mathbb F}_q}$. 

\item The stratum $S_{b} $: We take $b \in B(G_{\mathbb Q}, \mu)$ such that $\nu(b) = (s, \lambda_1^{p_1}, \lambda_2^{p_2})$ (note that we must have $\lambda_1p_1, \lambda_2p_2 \in \mathbb Z$, see Corollary 1.3 of \cite{Ra02}), such that 
\begin{itemize}
\item $p_1, p_2 \in \mathbb Z_{>0}$ and $p_1 + p_2 = n$,
\item $(p_1 \lambda_1, p_1) = 1$ and $(p_2\lambda_2, p_2) = 1$. 
\end{itemize}

 We use $B_K$ to denote the locally closed subscheme of $M_{K, \mathbb F_q}$ (defined over $\mathbb F_q$), such that $B_{K,\overline{\mathbb F}_q}$ is $S_b$ (see Theorem 3.6 of \cite{RaRi96}).
The Hecke correspondences on $Sh_K(G,X)$ may be restricted to the subvariety $\iota : B_K \rightarrow
  M_{K, \mathbb F_q}$, and we have an action of the algebra $\mathcal H(G(\mathbb A_f)//K)$ on the cohomology spaces $H^i_{et,c} (S_b , \iota^*\mathcal L)$. Furthermore, the space $H^i_{et,c} (S_b , \iota^*\mathcal L)$ carries an action of the Galois
group $\mathrm{Gal}(\overline{\mathbb F}_{q} /\mathbb F_{q} )$, which commutes with the action of $\mathcal H(G(\mathbb A_f)//K)$.

\item Local systems: Let $\zeta$ be an irreducible algebraic representation over $\overline{\mathbb Q}$ of $G_{\overline{\mathbb Q}}$, associated with an $\ell$-adic local system on $\bar S$ (where $\ell \not = p$ is a prime number). Let $\overline{\mathbb Q}_{\ell}$ be an algebraic closure of $\mathbb Q_{\ell}$ with an embedding $\overline{\mathbb Q} \subset \overline{\mathbb Q}_{\ell}$. We write $\mathcal L$ for the $\ell$-adic local system on $M_{K,\mathcal O_{E_{p}}}$ associated to the representation $\zeta \otimes \mathbb Q_{\ell}$ of $G_{\bar{\mathbb Q}_{\ell}}$.

\item Hecke algebra at $\infty$:  Let $\mathfrak g$ be the Lie algebra of
$G(\mathbb R)$, and let $K_{\infty}$ denote the stabilizer subgroup in $G(\mathbb R) $ of the morphism $h$. We use $f_{ \zeta, \infty}$ to denote an Euler–Poincar{\'e} function associated to $\zeta$ (as introduced in \cite{Clo85}, see also $\S 3$ of \cite{Kot92}). Note that $f_{ \zeta, \infty}$ has the following property: Let $\pi_{\infty}$ be a $(\mathfrak g, K_{\infty})$-
module occurring as the component at infinity of an automorphic representation $\pi$ of
$G$. Then the trace of $f_{ \zeta, \infty}$ against $\pi_{\infty}$ is equal to
$N_{\infty}\sum_{i=1}^{\infty}(-1)^idim(H^i(\mathfrak g, K_{\infty};\pi_{\infty}\otimes \zeta))$, where $N_{\infty}$ is a certain explicit constant as defined in $\S 3$ of \cite{Kot92}. 

\item Hecke operator away from $p$ and $\infty$: Let $f ^{p\infty} \in \mathcal H(G(\mathbb A)//K^p)$ be an arbitrary $K^p$-bi-invariant Hecke operator. 

\item Hecke operator at $p$: Let $f_p\in \mathcal H^{\mathrm{unr}}(G(\mathbb Q_p))$ denote the function whose Satake transform is given by $ X^{\alpha} \otimes q^{\frac{\alpha s(n-s)}2}\sum\limits_{\substack{I \subset {1,2, \ldots, n}\\|I| = s}} (X^I)^ {\alpha}$ where $X^I:=\prod_{i \in I} X_i $. Let $\mathbb Q_{p,\alpha}$ denote a unramified extension of $\mathbb Q_p$ of degree $\alpha$. Let $\phi_{\alpha} \in \mathcal H^{\mathrm{unr}}(G(\mathbb Q_{p,\alpha}))$ denote the function whose Satake transform is given by $ Y \otimes q^{\frac{\alpha s(n-s)}2}\sum\limits_{\substack{I \subset {1,2, \ldots, n}\\|I| = s}} (Y^I)$. Then $\phi_{\alpha}$ and $f_p$ are associated in the sense of base change (as defined in $\S3.2$ of \cite{Lab99}). 
\end{enumerate}

We use $\lambda$ to denote $(\lambda_1^{p_1}, \lambda_2^{p_2}) \in \mathbb R_{\geq }^n$, and let $s_i\in\mathbb Z_{>0}$ denote $\lambda_ip_i$. Let $\mathrm{Frob}_q$ denote the Frobenius endomorphism of $ S_{\lambda} $ mapping $x$ to $x^q$, then we have 

\begin{theorem} Let $\alpha \in \mathbb Z_{>0}$, then we have: 
\begin{enumerate}
    
\item If $n$ is odd, or $n$ is even and $(p_1, p_2) \not = (\frac n2,\frac n2), (\frac n2 \pm 1, \frac n2 \mp 1)$, then $$\mathrm{Tr}(\mathrm{Frob}^{\alpha}_q\times f^{p \infty},\sum\limits_{i=0}^{\infty}H^i_{et,c}(S_{b},\iota^*\mathcal L)) = |\mathrm{ker}^1|(\mathbb Q, G) \sum\limits_{\substack{\pi \subset \mathcal A(G)\\ \rho_{\pi,p} \text{ is of }\text{ type I}} } c_{\pi}\mathrm{Tr}( C_{\lambda}f_{n\alpha s}, \rho_{\pi,p}) $$
\item If $n$ is even and $(p_1, p_2) = (\frac n2, \frac n2)$ or $(p_1,p_2) = (\frac n2 \pm 1, \frac n2 \mp 1)$, then $$\mathrm{Tr}(\mathrm{Frob}^{\alpha}_q\times f^{p \infty},\sum\limits_{i=0}^{\infty}H^i_{et,c}(S_{b},\iota^*\mathcal L)) = |\mathrm{ker}^1|(\mathbb Q, G) \sum\limits_{\substack{\pi \subset \mathcal A(G)\\ \rho_{\pi,p} \text{ is of }\text{ type II}} } c_{\pi}\mathrm{Tr}( C_{\lambda}f_{n\alpha s}, \rho_{\pi,p}) $$

\end{enumerate}
where
\begin{itemize}
\item the constant $c_\pi$ equals $\zeta_{\pi}^{\alpha} \mathrm{Tr}(f^{p\infty},\pi^{p \infty})N_{\infty}\sum_{i=1}^{\infty}(-1)^idim(H^i(\mathfrak g, K_{\infty};\pi_{\infty}\otimes \zeta))$ such that $\zeta_{\pi}$ is a root of unity. 
\item by assumption, we have $G(\mathbb Q_p) \cong \mathbb Q_p^{\times} \times GL_n(\mathbb Q_p)$, we write $\pi_p \cong \sigma_{\pi,p} \otimes \rho_{\pi,p} $, such that $\sigma_{\pi,p}$ (resp. $\rho_{\pi,p}$) is a representation of $\mathbb Q_p^{\times}$ (resp. $GL_n(\mathbb Q_p)$). 
\item the truncated traces $\mathrm{Tr}( C_{\lambda}f_{n\alpha s}, \rho_{\pi,p})$ where $\rho_{\pi,p}$ are representations of $GL_n(\mathbb Q_p)$ of type I or type II (as defined in Definition 5.3) are computed in Proposition 5.3. 

\end{itemize}
\end{theorem}
\begin{proof}
By applying Theorem 5.4.5 of \cite{Fuji97}, we have $$\mathrm{Tr}(\mathrm{Frob}^{\alpha}_q\times f^{p \infty},\sum\limits_{i=0}^{\infty}H^i_{et,c}(S_{b},\iota^*\mathcal L)) = \sum\limits_{x \in \mathrm{Fix}^{S_{b}}_{\mathrm{Frob}^{\alpha}_q\times f^{p \infty}}(\overline{\mathbb F}_q)} \mathrm{Tr}(\mathrm{Frob}^{\alpha}_q\times f^{p \infty},\iota^*\mathcal L_x)$$ Repeat the proof of Proposition 10 of \cite{Kret11}, we have 

 $$\sum\limits_{x \in \mathrm{Fix}^{S_{b}}_{\mathrm{Frob}^{\alpha}_q\times f^{p \infty}}(\overline{\mathbb F}_q)} \mathrm{Tr}(\mathrm{Frob}^{\alpha}_q\times f^{p \infty},\iota^*\mathcal L_x) = |\mathrm{ker}^1(\mathbb Q, G)|\sum_{(\gamma_0;\gamma,\delta)} c(\gamma_0;\gamma,\delta) O_{\gamma}(f^{p\infty})TO_{\delta}(\phi_{\alpha})\mathrm{Tr}(\zeta_{\mathbb C}(\gamma_0))$$ where 
\begin{itemize}
\item The sum ranges over Kottwitz triples (as defined in $\S14$ of \cite{kot92p}) $(\gamma_0;\gamma,\delta)$ such that the Newton vector $\nu(\delta\sigma(\delta)\cdots\sigma^{[E_{\mathfrak p, \alpha}:\mathbb Q_p]-1}(\delta))\in \mathbb R^{n+1} $ equals $ (s,\lambda_1^{p_1},\lambda_2^{p_2})$ (where $\sigma$ is the Frobenius automorphism). 
\item $c(\gamma_0;\gamma,\delta)$ is a certain constant defined in $\S19$ of \cite{kot92p}.

\end{itemize}

Given $x \in G(\mathbb Q_{p,\alpha})$, let $N(x)$ denote $x\sigma(x)\cdots\sigma^{\alpha-1}(x).$ By the proof of Theorem 2.14 on pp.24-25 of \cite{Kret12b}, we have  $$|\mathrm{ker}^1(\mathbb Q, G)|\sum_{(\gamma_0;\gamma,\delta)} c(\gamma_0;\gamma,\delta) O_{\gamma}(f^{p\infty})TO_{\delta}(\phi_{\alpha})\mathrm{Tr}(\zeta_{\mathbb C}(\gamma_0)) $$
$$= |\mathrm{ker}^1(\mathbb Q, G)|\sum_{(\gamma_0;\gamma,\delta)} c(\gamma_0;\gamma,\delta) O_{\gamma}(f^{p\infty})TO_{\delta}(\chi^{\sigma}_{(s,\lambda_1^{p_1}, \lambda_2^{p_2})}\phi_{\alpha})\mathrm{Tr}(\zeta_{\mathbb C}(\gamma_0))$$ where $\chi^{\sigma}_{(s,\lambda_1^{p_1}, \lambda_2^{p_2})}$ is defined as follows: Let $x \in G(\mathbb Q_{p,\alpha})$, then $\chi^{\sigma}_{(s,\lambda_1^{p_1}, \lambda_2^{p_2})}(x) = 1$ if and only if $\nu(N(x)) = r (s, \lambda_1^{p_1}, \lambda_2^{p_2})$ for some $r \in \mathbb R_{>0}$. By repeating the argument in p509 of \cite{Kret11}, we have $$
 |\mathrm{ker}^1(\mathbb Q, G)|\sum_{(\gamma_0;\gamma,\delta)} c(\gamma_0;\gamma,\delta) O_{\gamma}(f^{p\infty})TO_{\delta}(\chi^{\sigma}_{(s,\lambda_1^{p_1}, \lambda_2^{p_2})}\phi_{\alpha})\mathrm{Tr}(\zeta_{\mathbb C}(\gamma_0))$$
 $$= |\mathrm{ker}^1(\mathbb Q, G)| \mathrm{Tr}(C_{(s, \lambda_1^{p_1}, \lambda_2^{p_2})}f_pf^{p\infty} f_{ \zeta, \infty},\mathcal A(G)),$$ where $\mathcal A(G)$ is the space of automorphic forms on $G$, and the function $C_{(s, \lambda_1^{p_1}, \lambda_2^{p_2})} \in \mathcal H(G(\mathbb Q_p))$ is defined in Definition 3.1. 

 Now we compute $\mathrm{Tr}(C_{(s, \lambda_1^{p_1}, \lambda_2^{p_2})}f_pf^{p\infty} f_{ \zeta, \infty},\mathcal A(G))$. We have $$\mathrm{Tr}(C_{(s, \lambda_1^{p_1}, \lambda_2^{p_2})}f_pf^{p\infty} f_{ \zeta, \infty},\mathcal A(G)) = \sum\limits_{\pi \subset \mathcal A(G)} \mathrm{Tr}(C_{(s, \lambda_1^{p_1}, \lambda_2^{p_2})}f_pf^{p\infty} f_{ \zeta, \infty},\pi)$$
 $$= \sum\limits_{\pi \subset \mathcal A(G)} \mathrm{Tr}(C_{(s, \lambda_1^{p_1}, \lambda_2^{p_2})}f_pf^{p\infty} f_{ \zeta, \infty},\pi)
 $$
  $$= \sum\limits_{\pi \subset \mathcal A(G)} \mathrm{Tr}(C_{(s, \lambda_1^{p_1}, \lambda_2^{p_2})}f_p,\pi_p) \mathrm{Tr}(f^{p\infty},\pi^{p\infty})\mathrm{Tr}(f_{ \zeta, \infty},\pi_{\infty}).  
 $$
 By the definition of $f_{ \zeta, \infty}$, we have $\mathrm{Tr}(f_{ \zeta, \infty},\pi_{\infty}) = N_{\infty}\sum_{i=1}^{\infty}(-1)^idim(H^i(\mathfrak g, K_{\infty};\pi_{\infty}\otimes \zeta))$. By assumption, we have $G(\mathbb Q_p) \cong \mathbb Q_p^{\times} \times GL_n(\mathbb Q_p)$, we write $ \pi_p \cong \sigma_{\pi,p} \otimes \rho_{\pi,p}$ where $\sigma_{\pi,p}$ (resp. $\rho_{\pi,p}$) is an admissible representation of $\mathbb Q_p$ (resp. $GL_n(\mathbb Q_p)$). By repeating the proof of Proposition 11 in \cite{Kret11}, we find that $\rho_{\pi,p}$ is the local component of a discrete automorphic representation of $GL_n(\mathbb A)$. 
 Also note that $$C_{(s, \lambda_1^{p_1}, \lambda_2^{p_2})}f_p = C_{(s, \lambda_1^{p_1}, \lambda_2^{p_2})}(X^\alpha \otimes f_{n \alpha s}) = X^{\alpha} \otimes C_{\lambda}f_{n\alpha s}$$ (we use $\lambda $ denote $(\lambda_1^{p_1}, \lambda_2^{p_2})$), then $$\mathrm{Tr}(C_{(s, \lambda_1^{p_1}, \lambda_2^{p_2})}f_p,\pi_p) = \mathrm{Tr}(X^{\alpha} \otimes C_{\lambda}f_{n\alpha s},\sigma_{\pi,p}\otimes \rho_{\pi,p}) = \mathrm{Tr}( X^{\alpha}, \sigma_{\pi,p})\mathrm{Tr}( C_{\lambda}f_{n\alpha s}, \rho_{\pi,p}).$$ By repeating the argument in pp.511-512 of \cite{Kret11}, we have that $\mathrm{Tr}( X^{\alpha}, \sigma_{\pi,p}) = \zeta_{\pi}^\alpha$ where $\zeta_{\pi} =\mathrm{Tr}( X, \sigma_{\pi,p}) $ is a root of unity. By Corollary 3.1, we have $\mathrm{Tr}( C_{\lambda}f_{n\alpha s}, \rho_{\pi,p})  = 0$ unless $\rho_{\pi,p}^{I_n} \not = 0$ (with $I_n$ defined before Corollary 3.1). Then by applying Theorem 2 in \cite{Kret11}, we find that $\rho_{\pi,p}$ is a semi-stable rigid representation of $GL_n(\mathbb Q_p)$. 
 
 By Proposition 3.3, we have $$\mathrm{Tr}( C_{\lambda}f_{n\alpha s}, \rho_{\pi,p}) = \mathrm{Tr}(C^{GL_{p_1}\times GL_{p_2}(\mathbb Q_p)}_{c}  (f_{p_1 \alpha s_1} \otimes f_{p_2 \alpha s_2}), (\rho_{\pi,p})_{N_{\lambda}}(\delta_{GL_n, P_{\lambda}}^{-\frac 12}) ).$$ By proposition 5.2, we have our theorem when $\alpha $ is big enough. And we may conclude our theorem for an arbitrary $\alpha \in \mathbb Z_{>0}$ by repeating the proof of Theorem 5.2.1 in \cite{Yachen24}. 
\end{proof}

\begin{remark}
Using the argument of Proposition 12 in \cite{Kret11}, we may derive a formula to compute the dimension of $S_b$ from Theorem 6.1. 
\end{remark}
\clearpage
\addcontentsline{toc}{section}{References}
\renewcommand\refname{References}
\bibliographystyle{plain}
\bibliography{ref}
\end{document}